\documentclass[10pt,usletter]{article}

\usepackage{amsmath,amssymb,makeidx,amsfonts,amsthm,latexsym,enumerate,cite,tikz}

\newtheorem{lemma}{Lemma}[section]
\newtheorem{theorem}{Theorem}[section]
\newtheorem{corollary}{Corollary}[section]
\newtheorem{proposition}{Proposition}[section]

\newtheorem*{theorem*}{Theorem}

\newtheorem{thmA}{Theorem}[theorem]

\newtheorem{corB}{Corollary}[theorem]

\newtheorem{corC}{Corollary}[theorem]

\newtheorem*{summary of results}{Summary of results}

\theoremstyle{definition}

\newtheorem{remark}{Remark}[section]

\begin{document}

\author{Karl-Olof Lindahl\\
Departamento de Matem\'atica y Ciencia de la Computaci\'on\\
Universidad de Santiago de Chile\\
Av. Las Sophoras 173, Estaci\'on Central, Santiago de Chile\\
Supported by MECESUP2 (PUC-0711) and DICYT (U. de Santiago de Chile)\\ 
\texttt{jk.linden3@gmail.com}}

\title{The size of quadratic $p$-adic linearization disks\footnote{Published in \textit{Advances in Mathematics}, 248:972--894, 2013}}

\date{}
\maketitle


\begin{abstract}
We find the exact radius of  linearization disks at indifferent
fixed points of
quadratic maps in $\mathbb{C}_p$. We also show that
the radius is invariant under power series perturbations.
Localizing all periodic orbits of these quadratic-like maps we
then show that periodic points are not the only obstruction for linearization. 
In so doing, we provide the first known examples in the dynamics of
polynomials over $\mathbb{C}_p$ 
where the boundary of the linearization disk does not contain any periodic point.
\end{abstract}

\vspace{1.5ex}\noindent {\bf Mathematics Subject Classification
(2010): 37F50, 37P20, 37P05, 11S15, 32H50} 

\vspace{1.5ex}\noindent {\bf Key words:} Small divisors; Linearization; $p$-adic numbers;
Ramification; Periodic points.

\section{Introduction}


Let $p$ be a prime and let $\mathbb{C}_p$
be the completion of an algebraic closure of the
field of $p$-adic numbers $\mathbb{Q}_p$.  
We study a $p$-adic analogue of the Siegel center problem in complex
dynamics. A power series
\begin{equation*}
f(x)=\lambda x + \langle x^2\rangle \in\mathbb{C}_p[[x]],
\quad |\lambda|=1, \textrm{ not a root of unity},
\end{equation*}
has an irrationally indifferent fixed point at $x=0$,
and is said to be analytically \emph{linearizable} at $x=0$
if there exists a convergent power series solution $H_f$
to the functional equation
\begin{equation*}
H_f\circ f\circ H_f^{-1}(x)=\lambda x.
\end{equation*}
For a large class of quadratic-like maps we find the exact radius of 
the corresponding linearization disk about $x=0$.
Localizing the periodic orbits of these maps, we then show that 
the convergence of $H_f$ stops \textit{before} the appearance
of any periodic point different from $x=0$. 
%
%
%
%
Our starting point is Yoccoz study of quadratic Siegel disks \cite{Yoccoz:1995}.
Let $\alpha$ be irrational and $\{p_n/q_n\}_{n\geq 0}$
be the approximants given by its continued fraction expansion.
The Brjuno series $B(\alpha)$ is defined by
$B(\alpha)=\sum_{n\geq 0}\log(q_{n+1})/q_n$.
Yoccoz proved that
$P_{\alpha}(z)=e^{2i\pi\alpha}z+z^2\in\mathbb{C}[z]$ is linearizable at $z=0$
if and only if $B(\alpha)<\infty$.
If $B(\alpha)<\infty$
and $C(\alpha)$ is the conformal radius of the Siegel disk, then there exist
universal constants $C_1$ and $C_2$ such that
$C_1<B(\alpha)+\log C(\alpha)<C_2$.
The lower bound is due to Yoccoz \cite{Yoccoz:1995}, and the upper bound
is due to Buff and Ch\'eritat \cite{BuffCheritat:2004}.
A possible obstruction for linearization is the presence of periodic points
different from zero. Indeed, Yoccoz proved that
if $\alpha$ is irrational and $B(\alpha)=\infty$, then
every neighborhood of $z = 0$ contains a periodic orbit of $P_{\alpha}$
different from $z = 0$. However, as shown by P\'erez-Marco
\cite[TheoremV.4.2]{Perez-Marco:1997}, in complex dynamics there exist
maps having no periodic point on the boundary of a Siegel disk.
Below, we present an analogue of the Yoccoz Brjuno function, $\tilde{r}(\lambda)$, 
that estimates the radius of $p$-adic linearization disks (Theorem \ref{thmA general estimate}).
Using the optimal bound obtained for $p$-adic quadratic maps
(Theorem \ref{thmA quadratic maps}),
we obtain an analogue of P\'erez-Marco's result on non-existence of periodic points on 
the boundary (Corollary \ref{corC no periodic boundary}),
providing the first  $p$-adic examples of its kind.


Let $p$ be a prime and denote by $\mathcal{O}_p$ the closed unit disk in $\mathbb{C}_p$. 
For any $\lambda\in\mathbb{C}_p$ with $|\lambda|=1$ but
$\lambda $ not a root of unity, we define
\begin{equation*}
\mathcal{G}_{\lambda}:=
\lambda x+x^2\mathcal{O}_p[[x]].
\end{equation*}
By the Non-Archimedean Siegel Theorem of
Herman and Yoccoz \cite{Herman/Yoccoz:1981},
$f\in\mathcal{G_{\lambda}}$
is always linearizable at $x=0$.
Given $f\in\mathcal{G_{\lambda}}$, we denote by $\Delta_f$ the corresponding
\emph{linearization disk},
\textit{i.e.} the maximal disk about the origin on which $f$ is
analytically conjugate to
$T_{\lambda}:x\to \lambda x$.
We denote by $r(f)$ the exact radius of $\Delta_f$ and
introduce a function $\tilde{r}=\tilde{r}(\lambda)$
that estimates $r(f)$ from below. 
There is an explicit formula for 
$\tilde{r}(\lambda)$, to be given in formula (\ref{definition tilde r})
in Section \ref{estimates of linearization disks}, with the
properties stated in Theorem  \ref{thmA general estimate}
and \ref{thmA quadratic maps} below.

\begin{thmA}
\label{thmA general estimate}
Let $f\in\mathcal{G}_{\lambda}$, then $r(f)\geq\tilde{r}(\lambda)$.
\end{thmA}

The proof is based on estimates of the coefficients of the
conjugacy function $H_f$ and an application of the Weierstrass Preparation Theorem.
As our main result, refining the estimates of $H_{f}$ for quadratic maps,
we obtain the exact size of quadratic linearization disks.

\begin{thmA}
\label{thmA quadratic maps}
Let $p\geq 3$ and $\lambda\in\mathbb{C}_p$ be not a root of unity and put
\[
P_{\lambda}(x):=\lambda x +x^2\in\mathbb{C}_p[x], \quad
\textrm{where  $1/p<|1-\lambda|<1$.}
\]
Then, the linearization disk $\Delta_{P_{\lambda}}$ is the open disk
$D_{r(P_{\lambda})}(0)$ 
where
$r(P_{\lambda})=|1-\lambda|^{-1/p}\tilde{r}(\lambda)$.
\end{thmA}

In view of Theorem \ref{thmA quadratic maps},
the general estimate $\tilde{r}(\lambda)$ is nearly optimal.
The work lies in
obtaining the optimal estimates of the coefficients of $H_{P_{\lambda}}$
carried out in Section \ref{estimates of linearization disks} and \ref{section quadratic case}.
To our knowledge, this is the first known case in $\mathbb{C}_p$
where the exact linearization disk is
known without having an explicit formula for the conjugacy $H_{f}$. In fact,
the radius $r(P_{\lambda})$ is invariant under power series perturbations.

\begin{corB}
\label{corB quadratic power series}
Let $p\geq 3$ and $f\in\mathcal{G}_{\lambda}$, with $1/p<|1-\lambda|<1$, 
be of the form
\begin{equation}\label{form II quadratic}
f(x)=\lambda x +\sum_{i=2}^{\infty}a_ix^i\in\mathcal{O}_p[[x]],
\quad \textrm{where  $|a_2|=1$ and  $|a_2^2-a_3|=1$.}
\end{equation}
Then, the linearization disk $\Delta_{f}$
equals the disk $\Delta_{P_{\lambda}}$ of radius
$r(P_{\lambda})=|1-\lambda|^{-1/p}\tilde{r}(\lambda)$.
\end{corB}

By Corollary  \ref{corB quadratic power series},
the exact radius $r(f)=|1-\lambda|^{-1/p}\tilde{r}(\lambda)$
depends only on $\lambda$ for $f$ in this family. 
Also note the we only impose an extra condition on the cubic term of $f$ for the generalization to hold.
In fact, this condition  is the
same as the condition for $f$ to be minimally ramified \cite{Rivera-Letelier:2003thesis}.
Using Corollary  \ref{corB quadratic power series}
and localizing all periodic orbits of $f$,
we prove the non-existence of periodic points on the boundary 
$\partial\Delta_{f}:=\{x\in\mathbb{C}_p: |x|=r(f)\}$.
\begin{corC}\label{corC no periodic boundary}
Let $p\geq 3$ and let $f$ be of the form (\ref{form II quadratic}).
Then,
the boundary
$\partial\Delta_{f}$ of the linearization disk does not contain any periodic point of $f$.
\end{corC}
\begin{remark}
The proof of Corollary \ref{corC no periodic boundary} induces a stronger version of the Corollary itself in the sense that
in Lemma \ref{lemma localization of periodic points} we localize all periodic orbits
for every $f(x)=\lambda x +a_2x^2+a_3x^3+\dots \in\mathcal{G}_{\lambda}$,
such that $0<|1-\lambda|<1$, $|a_2|=1$ and $|a_2^2-a_3|=1$, using that such maps are
minimally ramified \cite{Rivera-Letelier:2003thesis}.
\end{remark}
Corollary \ref{corC no periodic boundary},
similar in its flavor to the complex field
case result of P\'erez-Marco \cite[TheoremV.4.2]{Perez-Marco:1997},
reveals a new phenomenon in the dynamics of $p$-adic polynomials.
In contrast to our result obtained for quadratic-like maps,
the $p$-adic power functions, which are the only previously known examples where the exact size of the linearization disk is known,  
all have periodic points on the boundary of the
linearization disk \cite{Lubin:1994,Arrowsmith/Vivaldi:1994},
see figure \ref{figure periodic P} below.
For $g(x)=a_0+a_1x+\dots\in\mathcal{O}_p[[x]]$, denote by wideg$(g)$ the
\textit{Weierstrass degree}
of $g$, \textit{i.e.} the smallest number $d\in \{0,1,\dots\}\cup\{\infty\}$
such that $|a_d|=1$.
If $d$ is finite, then $d$ is the number of zeros of $g$ in the open unit disk,
counting multiplicity.
In fact,  in view of Rivera-Letelier's classification of $p$-adic Siegel disks
\cite{Rivera-Letelier:2003thesis}, every polynomial $f\in\mathcal{G}_{\lambda}$
with $2\leq$wideg$(f-\textrm{id})<\infty$, has infinitely many periodic points
in the open unit disk.
From this point of view, and the results obtained for power maps, it may be tempting to conjecture that a generic
$f\in\mathcal{G}_{\lambda}$, with $2\leq$wideg$(f-\textrm{id})<\infty$, has a periodic point on $\partial\Delta_f$.
Corollary \ref{corC no periodic boundary}
shows that such a conjecture is indeed false.

The paper is organized as follows.
In Section \ref{preliminaries} 
we present preliminaries concerning power series and the formal solution,
and in Section \ref{section geometry of unit sphere} we consider the
geometry of the unit sphere and the arithmetic of the multiplier
$\lambda$. 
Sections \ref{estimates of linearization disks} and 
\ref{section quadratic case} are devoted to 
explicit, and in several cases optimal, bounds
of the conjugacy function $H_f$,
from which we obtain Threorem \ref{thmA general estimate}
and \ref{thmA quadratic maps} respectively.
In Section \ref{section corollary A} we give a proof of Corollary
\ref{corC no periodic boundary}. 
Section \ref{section corollary A} ends with an example illustrating the main idea.

We end this introduction by some remarks given below.
For some additional references on non-Archimedean dynamics and its
relationship, similarities, and differences with respect to the
Archimedean theory of complex dynamics, see
\cite{Silverman:2007,
AnashinKhrennikov:2009,Khrennikov/Nilsson:2004,
KhrennikovSvensson:2007,
GhiocaTuckerZieve:2008} and references therein.

\subsection{Further remarks}

\begin{remark}
Our estimate of $r(f)$ in Theorem  \ref{thmA general estimate} extends results
obtained for
quadratic polynomials over $\mathbb{Q}_p$ by
Ben-Menahem \cite{Ben-Menahem:1988},
and Thiran, Verstegen, and Weyers \cite{Thiran/EtAL:1989},
and for certain polynomials
with maximal multipliers over the $p$-adic integers $\mathbb{Z}_p$
by Pettigrew, Roberts and Vivaldi \cite{Pettigrew/Roberts/Vivaldi:2001},
as well as results on small divisors by Khrennikov \cite{Khrennikov:2001a}.
\end{remark}

\begin{remark}\label{remark upper bound}
If $f\in\mathcal{G}_{\lambda}$ and $2\leq$wideg$(f-\textrm{id})<\infty$,
then we also have the upper bound $r(f)<1$, since in this case $f$ has at
least one periodic point $x\neq 0$ in the open unit disk.

\end{remark}

\begin{remark}\label{remark critical points}
There is a number of results on the existence of critical points and/or failing of
injectivity on the boundary of complex quadratic linearization disks
since the work of Herman \cite{Herman:1985}.
See \cite{CheritatRoesch:2011,Zhang:2012} for more results and references.
Another possible obstruction for linearization would be the presence of transient points,
\textit{i.e.} points on $\partial \Delta_f$ that escape $\partial \Delta_f$
after some iterations. But neither of these obstructions occur for
$f\in\mathcal{G}_{\lambda}$ with $2\leq$wideg$(f-\textrm{id})<\infty$;
by Remark \ref{remark upper bound}, $r(f)<1$, and since $f$ is isometric on the open 
unit disk, it is also isometric on  $\partial\Delta_f$.
Hence, it remains open whether or not
there is a `dynamical' obstruction for linearization on $\partial\Delta_f$
for $f$ of the form (\ref{form II quadratic}).
\end{remark}

\begin{remark}
In contrast to the complex field case \cite{Carleson/Gamelin:1991,Milnor:2000},
the boundary of the linearization disk
may not be contained in the topological closure of the post-critical set, see
Remark \ref{remark post-critical}.
\end{remark}

\begin{remark}
For irrationally indifferent fixed points, the conjugacy function is related to
the iterative Lie-logarithm, $f_*=\lim_{n\to\infty}(f^{\circ p^n}-\textrm{id})/{p^n}$.
Formally,  $H_f=\exp(\int\log(\lambda)/f_*)$.
See \cite{Lubin:1994,Li:1996c,Rivera-Letelier:2003thesis} and references therein
for studies of $f_*$ and its connection to the dynamics of $f$.
\end{remark}

\begin{remark}\label{remark siegel disks}
There are two possible definitions of a Siegel disk in the $p$-adic setting;
the linearization disk, or the maximal domain of quasi-periodicity.
In contrast to the complex field case, they do not coincide in the $p$-adic setting
since $p$-adic domains of quasi-periodicity
contain periodic points \cite{Rivera-Letelier:2003thesis}.
It is standard in $p$-adic dynamics to use the latter definition.
\end{remark}

\begin{remark}
In the multi-dimensional $p$-adic case, there exist multipliers $\lambda$ such that the corresponding
Siegel condition is violated and the conjugacy diverges
\cite{Herman/Yoccoz:1981,Viegue:2007}.
In fields of
positive characteristics, the convergence of the linearization series is far from trivial
even in the one-dimensional case \cite{Lindahl:2004,Lindahl:2010}.
For results in fields of caracteristic zero--equal characteristic case,
see \cite{Lindahl:2009eq}.
In the hyperbolic case $|\lambda|\neq1,0$, the linearization disk 
will in general be the maximal disk of injectivity \cite{LindahlZieve:2010}.
\end{remark}

\section{Mapping properties and the formal solution}\label{preliminaries}

By definition, $\mathbb{C}_p$ is a non-Archimedean field, 
\textit{i.e.} complete with respect to a non-trivial absolute value $|\cdot |$,
satisfying the  following
ultrametric triangle inequality:
\begin{equation}\label{sti}
|x+y| \leq  \max[|x|,|y|],\quad\text{for all $x,y\in \mathbb{C}_p$}.
\end{equation}
One useful consequence of ultrametricity is that for any
$x,y\in\mathbb{C}_p$ with $|x|\neq |y|$, the inequality (\ref{sti}) becomes an
equality.
The absolute value on $\mathbb{C}_p$ is normalized 
so that $|p|=1/p$.  For $x\in\mathbb{C}_p$ we denote
by $\nu(x)$ the \textit{valuation} of $x$, \textit{i.e.}
\[
\nu(x):=-\log_p|x|, \quad  \textrm{where $\nu(0):=\infty$}.
\]

Given an element
$x\in \mathbb{C}_p$ and real number $r>0$ we denote by $D_{r}(x)$ the open
disk of radius $r$ about $x$, by $\overline{D}_r(x)$ the closed
disk, and by $S_{r}(x)$ the sphere of radius $r$ about $x$.
Put $\mathbb{C}_p^*:=\mathbb{C}_p\setminus\{0\}$.
If $r\in|\mathbb{C}_p^*|$, \textit{i.e.} if there exist $a\in\mathbb{C}_p^*$
such that $|a|=r$,  we say that $D_{r}(x)$ and
$\overline{D}_r(x)$ are \emph{rational}.
If $r\notin |\mathbb{C}_p^*|$,
then we will call $D_{r}(x)=\overline{D}_r(x)$ an
\emph{irrational} disk.
Note that all disks are both open and closed as
topological sets, because of ultrametricity. However, as we will see
below,
power series distinguish between rational open, rational closed, and irrational disks.

\subsection{Univalent mappings}\label{section non-Archimedean power series}

In this paper we consider univalent mappings fixing zero.
Let $a>0$ be a real number and put
\begin{equation*}
\mathcal{U}_a:=\left\{ \sum_{i\geq 1}a_ix^{i}\in\mathbb{C}_p[[x]]: |a_1|=1 \textrm{, and }
1/\sup_{i\geq 2} |a_i|^{1/(i-1)}\geq a \right\}.
\end{equation*}
Then, $h\in\mathcal{U}_a$ converges on the open disk $D_{\rho(h)}(0)$ of radius
\begin{equation}\label{radius of convergence}
\rho(h) := \frac{1}{\limsup |a_i| ^{1/i}}\geq a.
\end{equation}
A power series $h\in\mathcal{U}_a$ converges on the sphere
$S_{\rho(h)}(0)$ if and only if
\[
\lim_{i\to\infty}|a_i| \rho(h) ^i=0.
\]
It is well known that all $h\in\mathcal{U}_a$ are univalent and isometric on the open disk $D_a(0)$. 
For future reference, we state these facts in the following proposition.

\begin{proposition}\label{proposition isometry}
Let $h\in\mathcal{U}_a $. Then, $h:D_a(0)\to D_a(0)$ is a bijective isometry.
In particular, if $f\in \mathcal{G}_{\lambda}$, then
$f:D_{1}(0)\to D_{1}(0)$ is bijective and isometric.
\end{proposition}
In fact, this result can be proven using the following generalization of the Weierstrass Preparation
Theorem, see Benedetto \cite[Lemma 2.2]{Benedetto:2003a}, that we will utilize later on, 
proving Theorem \ref{thmA general estimate} and Theorem
\ref{thmA quadratic maps},
respectively.

\begin{proposition}\label{proposition one-to-one}

Let $h\in\mathcal{U}_a$. Then the following two statements hold.

\begin{enumerate}
 \item Suppose that $h$ converges on the rational closed disk
  $\overline{D}_R(0)$. Let $0<r\leq R$ and suppose that
  \begin{equation}\label{ck inequality one-to-one}
   |c_i|r^i\leq
   r\quad \text{ for all } i\geq 2 .
  \end{equation}
 Then, $h:D_r(0)\to D_r(0)$ is a bijection.
 Furthermore, if
 \[
 d = \max\{i\geq 1:|c_i|{r}^i=r\},
 \]
 then $h$ maps the  closed disk $\overline{D}_{r}(0)$ onto itself
 exactly $d$-to-1 (counting
 multiplicity).

 \item Suppose that $h$ converges on the rational open disk
  $D_R(0)$.
  Let $0<r\leq R$ and suppose that
  \[
   |c_i|r^i \leq
   r\quad \text{ for all } i\geq 2 .
  \]
  Then, $h: D_{r}(0)\to D_r(0)$ is bijective.

\end{enumerate}

\end{proposition}

\subsection{The formal solution}

With $f\in\mathcal{G}_{\lambda}$, we associate
a conjugacy function $H_f\in\mathbb{C}_p[[x]]$ 
such that $H_f\circ f(x)=\lambda H_f(x)$,
and normalized so that $H_f(0)=0$ and $H_f'(0)=1$.
The lower bound for the size of linearization disks for
$f\in\mathcal{G}_{\lambda}$ is based on the following lemma.

\begin{lemma}\label{lemma bk estimate indiff}
Let $f\in \mathcal{G}_{\lambda}$. 
Then, there exists a conjugacy function 
$H_f(x)=\sum_{k=1}^{\infty}b_kx^k$
in  $\mathbb{C}_p[[x]]$.
The coefficients of $H_f$ satisfy, $b_1=1$ and
\begin{equation}\label{b_k estimate by prod 1-lambda n}
|b_k|\leq \left( \prod_{n=1}^{k-1}|1-\lambda ^n| \right
)^{-1},
\end{equation}
for all $k\geq 2$.
\end{lemma}

\begin{proof}
The existence of $H_f$ was proven in \cite{Herman/Yoccoz:1981}.
Let $f(x)=\sum_{i=1}^{\infty}a_ix^{i}\in\mathcal{G}_{\lambda}$
and let $H_f(x)=\sum_{k=1}^{\infty}b_kx^k$.
Let $k\geq 2$. Solving the equation $H_f\circ f(x)=\lambda H_f(x)$ 
for the coefficients of $x^k$ we obtain
\begin{equation}\label{bk-equation}
b_k=\frac{1}{\lambda(1-\lambda^{k-1})}\sum_{l=1}^{k-1}b_lC_l(k),
\end{equation}
where $C_l(k)$ is the coefficient of the $x^k$-term in
\[
(a_1 x +a_2x^2+\dots )^l.
\]
Let $\mathbb{N}$
be the set of non-negative integers.
For $k\geq 2$ and $l\in[1,k-1]$, put
\begin{equation*}
A_l(k):=\{(\alpha_1,\dots,\alpha_k)\in\mathbb{N}^k: \sum_{i=1}^{k}\alpha_i=l,
\textrm{ and } \sum_{i=1}^{k}i\alpha_i=k\}.
\end{equation*}
It follows by the Multinomial Theorem that
\begin{equation*}
C_l(k)=\sum_{(\alpha _1,\alpha _2,\dots,\alpha _k)\in A_l(k)}\frac{l!}{\alpha_1!\cdots
\alpha_k!}\lambda^{\alpha_1}a_2^{\alpha_2}\cdots a_k^{\alpha_k}.
\end{equation*}
Note that the factorial factors,
$l!/(\alpha_1!\dots \alpha_k!)$, are integers and
thus of absolute value less than or equal to $1$. Also recall that by definition,
$|a_i|\leq 1$.
It follows by ultrametricity that $|C_l(k)|\leq 1$ and
\[
|b_k|\leq \left( \prod_{n=1}^{k-1}|1-\lambda ^n| \right )^{-1},
\]
for all integers $k\geq 2$ as required.
\end{proof}


\section{Geometry of the unit sphere and the roots of unity}\label{%
section geometry of unit sphere}

Let $\Gamma$ be the group of all roots of unity in $\mathbb{C}_p$.
It has the subgroup $\Gamma_u$ ($u$ for unramified),
where
\begin{equation*}
 \Gamma_u:=\{\xi\in\mathbb{C}_p:\textrm{ } \xi ^m=1\textrm{
for some integer $m\geq 1$ not divisible by $p$}\}.
\end{equation*}

\begin{proposition}
\label{proposition gamma u}
The unit sphere $S_1(0)$ in $\mathbb{C}_p$ decomposes into the
disjoint union
\[
S_1(0)=\bigcup _{\xi \in \Gamma_u}D_1(\xi).
\]
In particular, $\Gamma_u\cap D_1(1)=\{ 1\}$ and consequently
$|1-\xi|=1$ for all $\xi\in\Gamma_u$ such that $\xi \neq 1$.
To each $\lambda \in S_1(0)$
there are unique $\xi\in \Gamma_u$ and $h\in D_1(1)$ such that
$\lambda =\xi h$.
\end{proposition}

A proof is given in 
\cite[p. 103]{Schikhof:1984}.
The other important subgroup of $\Gamma$ is $\Gamma_r$ ($r$ for
ramified), where
\begin{equation*}
\Gamma_r:=\{\zeta\in\mathbb{C}_p:\textrm{ }\zeta^{p^s}=1 \textrm{
for some integer $s\geq 0$}\}.
\end{equation*}
By elementary group theory $\Gamma _u\cap\Gamma _r =\{1\}$ and
$\Gamma$ is the direct product
$\Gamma =\Gamma_r\cdot\Gamma_u$.
Most importantly, for each $s\geq 1$ the primitive $p^s$-roots of unity
in $\Gamma_r $ are located in the sphere about $x=1$ of
radius
\begin{equation*}
R(s):=
p^{-\frac{1}{p^{s-1}(p-1)}}.
\end{equation*}
We also put $R(0):=0$.
Note that $R(s)$ is strictly increasing with $s$.
For integers $s\geq 0$, we will denote by $\mathcal{A}_s$
the group of $p^s$-roots of unity in $\Gamma_r$ and by
$\mathcal{B}_s$ the set $\mathcal{A}_s\setminus\mathcal{A}_{s-1}$.
We have the following proposition.

\begin{proposition}[See Theorem 8.9 in \cite{Escassut:1995}]
\label{proposition gamma r}

For integers $s\geq 1$, $\mathcal{B}_s$ consists of $p^s-p^{s-1}$
roots of order $p^s$, which all lie in $S_{R(s)}(1)$.
Moreover, for every $\zeta\in\mathcal{B}_s$,
we have $\mathcal{B}_s\cap D_{R(s)}(\zeta)=\zeta\mathcal{A}_{s-1}$.
\end{proposition}

The latter statement stems from the fact that multiplication by a
root of unity $\zeta$ induces an isometry of $\mathbb{C}_p$.
Moreover, by ultrametricity, for every $r>0$ and $\zeta\in D_{r}(1) $ we have
$D_{r}(1)=D_r(\zeta)$.
As a consequence,
the arrangement of roots of unity around $\zeta$ looks exactly
the same as the arrangement of roots of unity around $1$.
By the two propositions above and ultrametricity,
we have the following lemma 
(see also figure \ref{figure dynamics of lambda}).

\begin{lemma}\label{lemma closest root of unity}
Let $\lambda \in S_1(0)$ be not a root of unity.
Then,
$m=\min\{n\in \mathbb{Z}:n\geq 1, |1-\lambda ^n|<1 \}$
%
is well defined and not divisible by $p$.
Let $\gamma_0\in\Gamma_r$ be minimizing $|\gamma_0-\lambda^m|$
and let $t\geq 0$ be the integer such that
\[
R(t)\leq |\gamma_0-\lambda^m|<R(t+1).
\]
Then, for every $\gamma\in\gamma_0\mathcal{A}_t$
we have $|\gamma-\lambda^m|=|\gamma_0-\lambda^m|$
and for every integer $j\geq t+1$ and 
$\zeta\in S_{R(j)}(\gamma_0)\cap\Gamma_r$
we have $|\zeta-\lambda^m|=R(j)$. Moreover, 
the cardinality $\#S_{R(j)}(\gamma_0)\cap\Gamma_r=p^j-p^{j-1}$.
In particular, if $\lambda^m\notin S_{R(s)}(1)$ for any $s\geq 1$,
then we may assume $\gamma_0=1$. 
\end{lemma}

\subsection{Arithmetic of the multiplier}\label{section arithmetic of the multiplier}

\begin{lemma}
\label{lemma distance 1-lambda m char 0,p}
Let $\lambda \in \mathbb{C}_p$ be not a root of unity and $|\lambda|=1$.
Let $m$, $\gamma_0$ and $t$ be as
in Lemma \ref{lemma closest root of unity}.
Let $s\geq0$ be the integer for which
$ R(s)\leq |1-\lambda ^m|< R(s+1) $.
Then, for integers $n\geq 1$:

\begin{enumerate}

\item if $m$ does not divide $n$, we have $|1-\lambda^n |=1$,

\item if  $m$ is a divisor of $n$ and $\nu(n)<s$,
we have
$\left |1 -\lambda ^n \right |=|1-\lambda ^m|^{p^{\nu(n)}}$,

\item if $m$ is a divisor of $n$ and  $\nu(n)\geq s$, we have
$\left |1 -\lambda ^n \right |=
p^t|n||\gamma_0-\lambda ^m|^{p^t}$.
In particular, if $\lambda^m\notin S_{R(s)}(1)$,
then $\gamma_0=1$ and $t=s$.
\end{enumerate}

\end{lemma}

\begin{remark}
In the third statement, $|n|p^t\leq 1$ since we assume that $t\leq s\leq \nu(n)$ in this case.
\end{remark}


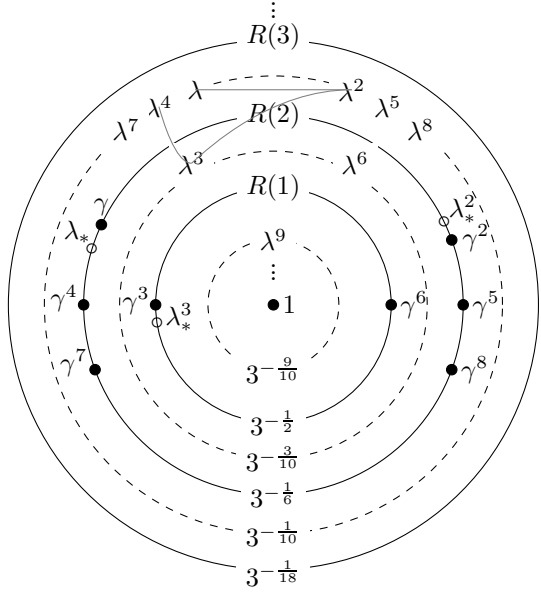
\begin{figure}
\begin{center}
\begin{tikzpicture}[scale=0.87]

\node[right,fill=white] at (0,0) {$1$};
\draw[fill] (0:0) circle (0.08);

\draw[dashed] (0,0) circle (1cm);

\node[fill=white] at (0,-1) {$3^{-\frac{9}{10}}$};

\draw (0,0) circle (1.8cm);
\node[fill=white] at (0,1.8) {$R(1)$};
\node[fill=white] at (0,-1.8) {$3^{-\frac{1}{2}}$};

\draw[fill] (0:1.8) circle (0.08);
\draw[fill] (180:1.8) circle (0.08);

\draw[dashed] (0,0) circle (2.35cm);
\node[fill=white] at (0,-2.35) {$3^{-\frac{3}{10}}$};

\draw (0,0) circle (2.9cm);
\node[fill=white] at (0,2.9) {$R(2)$};
\node[fill=white] at (0,-2.9) {$3^{-\frac{1}{6}}$};

\draw[fill] (20:2.9) circle (0.08);
\node[right] at (20.5:2.93) {$\gamma^2$};
\node at (26:2.9) {$\circ$};
\node[right] at (30:2.93) {$\lambda_*^2$};

\draw[fill] (155:2.9) circle (0.08);
\node[above] at (155:2.9) {$\gamma$};
\node at (163:2.9) {$\circ$};
\node[above] at (165:3.1) {$\lambda_*$};

\draw[fill] (180:2.9) circle (0.08);
\node[left] at (178:2.85) {$\gamma^4$};
\node at (189:1.8) {$\circ$};
\node[right] at (186:1.8) {$\lambda_*^3$};

\node[left] at (178:1.75) {$\gamma^3$};

\draw[fill] (200:2.9) circle (0.08);
\node[below] at (190:3.1) {$\gamma^7$};

\draw[fill] (340:2.9) circle (0.08);
\node[right] at (340:2.9) {$\gamma^8$};

\draw[fill] (360:2.9) circle (0.08);
\node[right] at (360:2.9) {$\gamma^5$};
\node[right] at (360:1.8) {$\gamma^6$};

\draw[dashed] (0,0) circle (3.5cm);
\node[fill=white] at (0,-3.5) {$3^{-\frac{1}{10}}$};
\node[fill=white] at (110:3.5) {$\lambda$};
\node[fill=white] at (70:3.5) {$\lambda^2$};
\draw[gray] (110:3.5) parabola (70:3.5);
\node[fill=white] at (120:2.5) {$\lambda^3$};
\draw[gray] (70:3.5) parabola (120:2.5);
\node[fill=white] at (120:3.5) {$\lambda^4$};
\draw[gray] (120:2.5) parabola (120:3.5);
\node[fill=white] at (60:3.5) {$\lambda^5$};

\node[fill=white] at (60:2.5) {$\lambda^6$};

\node[fill=white] at (130:3.5) {$\lambda^7$};

\node[fill=white] at (50:3.5) {$\lambda^8$};

\node[fill=white] at (90:1) {$\lambda^9$};

\node at (90:0.4) {.};
\node at (90:0.5) {.};
\node at (90:0.6) {.};

\draw (0,0) circle (4.05cm);
\node[fill=white]  at (0,4.1) {$R(3)$};
\node[fill=white]  at (0,-4.1) {$3^{-\frac{1}{18}}$};

\node at (0,4.4) {.};
\node at (0,4.5) {.};
\node at (0,4.6) {.};

\end{tikzpicture}
\end{center}

\caption{(Illustration of Lemmas \ref{lemma closest root of unity} and 
\ref{lemma distance 1-lambda m char 0,p}).
Let $\gamma$ be a primitive $9$-root
of unity.
The filled dots illustrate the roots of unity of order $1$, 
$3$,  and $3^2$
respectively.
Let $\lambda_*=\gamma+3\in\mathbb{C}_3$.
In this case
$|1-\lambda_*|=|1-\gamma|=R(2)$ so that
 $s(\lambda_*)=2$, 
 $m(\lambda_*)=1$,
$\gamma_0(\lambda_*)=\gamma$, 
with $|\gamma_0-\lambda|=1/3$ so that 
$t(\lambda_*)=0$. 
The orbit of $\lambda_*$ will be infinite
but on the larger scale it will resemble that of $\gamma$.
Also note that $|\gamma-\gamma^4|=|\gamma||1-\gamma^3|=R(1)$, 
whereas for example
$|\gamma-\gamma^2|=R(2)$. By ultrametricity we have 
$S_{R(1)}(\gamma_0)\cap\Gamma_r=
\{\gamma^4,\gamma^7\}$,
$S_{R(2)}(\gamma_0)\cap\Gamma_r=
\{1,\gamma^3,\gamma^6,\gamma^2,\gamma^5,\gamma^8\}$, 
and 
$S_{R(3)}(\gamma_0)\cap\Gamma_r=
S_{R(3)}(1)\cap\Gamma_r=\mathcal{B}_3$.
The figure also illustrates the first nine points in the
orbit of $\lambda=1+3^{\frac{1}{10}}\in\mathbb{C}_3$.
For higher iterates, $\lambda^{10}$ will be close to $\lambda$,
and $\lambda^{11}$ close to $\lambda^{2}$ and so forth.
In this case $R(2)<|1-\lambda|<R(3)$ and hence $s(\lambda)=2$,
$m(\lambda)=1$, $\gamma_0(\lambda)=1$ and $t(\lambda)=2$.
}\label{figure dynamics of lambda}
\end{figure}


\begin{proof}[Proof of Lemma \ref{lemma distance 1-lambda m char 0,p}]
Part 1 is a direct consequence of Lemma \ref{lemma closest root of unity}.
To prove 2 and 3
it is enough to consider the case $m=1$. Hence, we shall assume that
$\lambda\in D_1(1)$.
The proof is based on the factorization of the polynomial
$\lambda ^n-1$ from which we obtain
\begin{equation*}
\left |\lambda ^n - 1 \right |=
\prod _{\theta ^n=1}\left |\lambda - \theta \right |.
\end{equation*}
As noted in the previous section,
the set of roots of unity in $\mathbb{C}_p$ is given by the product
$\Gamma=\Gamma_r\cdot\Gamma_u$. This representation enables us to
write $\left |\lambda ^n - 1 \right |$ in the form
\begin{equation}\label{splitedfactorequation}
\left |\lambda ^n - 1 \right |=
\prod_{\zeta ^n=1}\left |\lambda - \zeta \right |
\prod _{(\zeta\xi )^n=1}\left |\lambda - \zeta\xi \right |,
\end{equation}
where $\zeta\in \Gamma_r$ and $\xi\in\Gamma_u\setminus\{1\}$.
Recall that we assume that $\lambda\in D_1(1)$. In view of Proposition 
\ref{proposition gamma u}, for $\xi\neq 1$ we have $\zeta\xi\notin D_1(1)$ 
and consequently
$|\lambda - \zeta\xi|=1$.
Moreover, for $n=ap^{\nu(n)}$, we
have that $\zeta^n=1$ if and only if $\zeta^{p^{\nu(n)}}=1$. It
follows that (\ref{splitedfactorequation}) can be reduced to
\begin{equation}\label{simplifiedfactorequation}
\left |\lambda ^n - 1 \right |=\prod
_{\zeta\in\mathcal{A}_{\nu(n)}}\left |\lambda - \zeta \right |.
\end{equation}
We distinguish between two cases given below.

\textbf{Case I:} $s\geq 1$ and $\nu(n)\in[0,s-1]$.
In this case, for every $\zeta\in \mathcal{A}_{\nu(n)}$ we have
\[
|1 - \zeta|<R(s)\leq |\lambda -1|.
\]
Consequently, by ultrametricity
$|\lambda - \zeta| =|(\lambda-1)+(1-\zeta)|=|\lambda-1|$.
As the cardinality $\#\mathcal{A}_{\nu(n)}=p^{\nu(n)}$, we obtain
the desired result using (\ref{simplifiedfactorequation}).

\textbf{Case II:} $s\geq 0$ and $\nu(n)\geq s$.
We first consider the case $\gamma_0=1$ so that $t=s$.
In this case $\mathcal{A}_s\subseteq \mathcal{A}_{\nu(n)}$ and hence
\begin{equation}\label{simplifiedfactorequation_split}
\prod
_{\zeta\in\mathcal{A}_{\nu(n)}}\left |\lambda - \zeta \right |=
\prod
_{\zeta\in\mathcal{A}_{s}}\left |\lambda - \zeta \right |
\prod_{s'=s+1}^{\nu(n)}
\left (
\prod_{\zeta\in\mathcal{B}_{s'}}\left |\lambda - \zeta \right |
\right ),
\end{equation}
where the product over $s'$ is set to be identically one if $\nu(n)=s$.
Recall that by assumption $\gamma_0=1$ and $t=s$, so that
$|\lambda -\zeta|=|\lambda-1|$ for all $\zeta\in \mathcal{A}_{s}$ and
\begin{equation}\label{equation A_s}
\prod
_{\zeta\in\mathcal{A}_{s}}\left |\lambda - \zeta \right |=|\lambda -1|^{p^s}.
\end{equation}
Moreover, for every integer $j>s $ and
$\zeta\in\mathcal{B}_{j}$ we
have $|\lambda-1|<R(s+1)\leq |1-\zeta|$ and hence
\[
|\lambda-\zeta|=|(\lambda-1)-(1-\zeta)|=|1-\zeta|=R(j).
\]
Therefore, in view of Proposition \ref{proposition gamma r},
\begin{equation*}
\prod_{s'=s+1}^{\nu(n)}
\left (
\prod_{\zeta\in\mathcal{B}_{s'}}\left |\lambda - \zeta \right |
\right )
=
\prod_{s'=s+1}^{\nu(n)}R(s')^{p^{s'}-p^{{s'}-1}}=p^{-(\nu(n)-s)}.
\end{equation*}
Combined with (\ref{simplifiedfactorequation}),
(\ref{simplifiedfactorequation_split}), and (\ref{equation A_s})
this implies the desired result.
This completes the proof in the case $\gamma_0=1$.

As to the remaining case $s\geq 1$ and $\gamma_0\in \mathcal{B}_s$.
Recall that as noted after Proposition \ref{proposition gamma r},
the arrangement of roots of unity about $\gamma_0$ looks
exactly like the arrangement of
roots of unity near $1$.
It follows by Lemma \ref{lemma closest root of unity}
that the arguments used above
in case II, carry over to the present case
replacing $1$ by
$\gamma_0$, $s$ by $t$, $\mathcal{A}_s$ by $\gamma_0\mathcal{A}_t$,
and $\mathcal{B}_j$ by $S_{R(j)}(\gamma_0)\cap\Gamma_r$.
This completes the proof of the lemma.
\end{proof}


\newpage

\section{General estimate of linearization disks}\label{estimates of linearization disks}

Throughout the rest of this paper let 
$m$, $s$, $\gamma_0$ and $t$ 
be as in Lemma \ref{lemma distance 1-lambda m char 0,p}
and put
\begin{equation}\label{definition tilde r}
\tilde{r}(\lambda):=R(s+1)^{\frac{1}{m}}p^{-\frac{s-t}{mp^s}}
|1-\lambda^m|^{s\frac{p-1}{mp} }
|\gamma_0-\lambda ^m|^{1/mp^{s-t}}.
\end{equation}
This section is devoted to prove Theorem
\ref{theorem general estimate} stated below,
from which we obtain Theorem \ref{thmA general estimate}.

\begin{theorem}\label{theorem general estimate}
Let $f\in \mathcal{G}_{\lambda}$.
Then, the linearization disk $\Delta_f\supseteq
D_{\tilde{r}(\lambda)}(0)$.
Moreover, if the conjugacy $H_f$ converges on the closed disk
$\overline{D}_{\tilde{r}(\lambda)}(0)$, then $\Delta_f\supseteq
\overline{D}_{\tilde{r}(\lambda)}(0)$.
\end{theorem}

\begin{remark}
Let $|\gamma_0 -\lambda ^m|$ be fixed. Then, the estimate
$\tilde{r}(\lambda)\to 1$
as $m$ or  $s$ goes to infinity. The latter case follows from the fact that
 both $R(s+1)$ and $|1-\lambda^m|^s$ approach one, as $s$ goes to infinity.
On the other hand,
if $s$ and $m$ are fixed, then $\tilde{r}(\lambda) \to 0$ as
$|\gamma_0 -\lambda^m|\to 0$.
\end{remark}

For any non-negative real number $\alpha$, let $\lfloor \alpha \rfloor$
denote the integer part of $\alpha$.  
The main ingredient in the proof of 
Theorem \ref{theorem general estimate}
is Lemma \ref{lemma theorem general estimate}
below, which also plays a fundamental role
in the analysis of the quadratic case in the following section.

\begin{lemma}\label{lemma theorem general estimate}
Let $\lambda\in S_1(0)$ be not a root of unity and let $k\geq 2$ be an integer.
Then, 
\begin{equation}\label{product ineq general case}
\left (\prod_{n=1}^{k-1}|1-\lambda ^n| \right )^{-1}
\leq
R(s+1)^{\frac{1}{m}}
|\gamma_0-\lambda^m|^{p^t
(\frac{k-1}{mp^s}-\left\lfloor\frac{k-1}{mp^s}\right\rfloor)
}
\tilde{r}(\lambda)^{-(k-1)}.
\end{equation}
Second, if $(k-1)/mp^{s}$ is a non-negative integer power of $p$, then
\begin{equation}\label{prod eq p power case}
\left (\prod_{n=1}^{k-1}|1-\lambda ^n| \right )^{-1}
=
p^{-\frac{1}{p-1}}
\tilde{r}(\lambda)^{-(k-1)}.
\end{equation}
Third,
if $m=1$ and $p^s\nmid k-1$ (so that $s\geq 1$), then
\begin{equation}\label{prod ineq general case m=1 pnmid k-1}
\left( \prod_{n=1}^{k-1}|1-\lambda ^n| \right )^{-1}
\leq
R(s+1)
|1-\lambda|^{\frac{k-1}{p}-\left\lfloor\frac{k-1}{p}\right\rfloor}
\tilde{r}(\lambda)^{-(k-1)}.
\end{equation}
\end{lemma}
\begin{remark}
Concerning the proof of Theorem \ref{thmA quadratic maps}
that will be given in the next section, 
the equality (\ref{prod eq p power case})
is essential in proving that
for a quadratic map $P_{\lambda}$, the conjugacy $H_{P_{\lambda}}$
diverges on the sphere of radius
$\tau=|1-\lambda|^{-1/p}\tilde{r}(\lambda)$. The
extra factor containing $|1-\lambda|$ 
in the estimate
(\ref{prod ineq general case m=1 pnmid k-1})
is essential in the proof that
in the quadratic case $H_{P_{\lambda}}$
is one-to-one on $D_{\tau}(0)$.

\end{remark}

\begin{remark}\label{remark strict}
Let $k\geq 2$. Then, in view of (\ref{product ineq general case}), 
we have 
\[
\left (\prod_{n=1}^{k-1}|1-\lambda^n| \right)^{-1} <\tilde{r}(\lambda)^{-(k-1)}.
\]
\end{remark}

The proof of Lemma \ref{lemma theorem general estimate} 
is below, after the following two lemmas.
\begin{lemma}[See Lemma 25.5 in 
\cite{Schikhof:1984}]\label{lemma order of p in factorial}
Let $n\geq 1$ be an integer and let $S_n$ be
the sum of the coefficients in the $p$-adic expansion of $n$.
Then,
\begin{equation}\label{equation order of p in factorial}
\nu(n!)=\frac{n-S_n}{p-1}\leq \frac{n-1}{p-1},
\end{equation}
with equality if $n$ is a power of $p$. Consequently, for all
integers $a\geq 1$,
\begin{equation}\label{limit order of p in factorial}
\lim_{n\to\infty}\frac{\nu(\left\lfloor \frac{n}{a}\right \rfloor!)} {n} 
=
\frac{1}{a(p-1)}.
\end{equation}
\end{lemma}

Throughout the rest of this section we fix an integer $k\geq2$ and
for any integer $z\geq 1$
we put
\[
\delta(z):=\frac{k-1}{z}-\left\lfloor \frac{k-1}{z}\right\rfloor.
\]
So defined 
$0\leq \delta <1-1/z$ and 
$\delta(z)=0$ if and 
only if $z$ is a divisor of $k-1$.

\begin{lemma}\label{lemma Sigma}
Let $s\geq 1$ be an integer and put
\[
\Sigma:=\sum_{j=0}^{s-1}
\left(
\left\lfloor\frac{k-1}{mp^j}\right\rfloor 
-\left\lfloor\frac{k-1}{mp^{j+1}}\right\rfloor
\right)p^j.
\]
Then,
\begin{equation}\label{Sigma_1 total}
\Sigma=
s\frac{k-1}{m}\frac{p-1}{p}
+
\delta(mp^s)p^{s-1}
-\delta(m)
-\sum_{s'=1}^{s-1}\delta(mp^{s'})
(p^{s'}-p^{s'-1}),
\end{equation}
where the sum over $s'$ is set to be identically zero if $s=1$.
\end{lemma}

\begin{proof}
Rearranging the terms of $\Sigma$, we obtain
\[
\Sigma=\left\lfloor\frac{k-1}{m}\right\rfloor
-
\left\lfloor\frac{k-1}{mp^s}\right\rfloor p^{s-1}
+
\sum_{s'=1}^{s-1}\left\lfloor\frac{k-1}{mp^{s'}}\right\rfloor
(p^{s'}-p^{s'-1}).
\]
Consequently, using $\lfloor (k-1)/mp^{s'}\rfloor=(k-1)/mp^{s'}-\delta(mp^{s'})$, 
we have
\begin{equation*}
\Sigma =
\left\lfloor\frac{k-1}{m}\right\rfloor -
\left\lfloor\frac{k-1}{mp^s}\right\rfloor p^{s-1}
+(s-1)\frac{k-1}{m}\frac{p-1}{p}-\sum_{s'=1}^{s-1}\delta(mp^{s'})
(p^{s'}-p^{s'-1}),
\end{equation*}
which implies (\ref{Sigma_1 total}) as required.
\end{proof}

\newpage

\begin{proof}[Proof of Lemma \ref{lemma theorem general estimate}]
Suppose that $s=0$. In this case $\gamma_0=1$ and $t=0$ so
$\tilde{r}(\lambda)$ reduces to
\[
\tilde{r}_1(\lambda):=p^{-\frac{1}{m(p-1)}}|1-\lambda ^m|^{\frac{1}{m}}.
\]
By Lemma \ref{lemma distance 1-lambda m char 0,p} we then have
\begin{equation}\label{distance 1- lambda n case 1}
\left |1 -\lambda ^n \right |= \left\{\begin{array}{ll}
1, & \textrm{if \quad $m\nmid n$,}\\
|n||1-\lambda ^m|, & \textrm{if \quad $m\mid n$.}
\end{array}\right.
\end{equation}
Let $N:=\lfloor (k-1)/m \rfloor$.
Then, from (\ref{distance 1- lambda n case 1}) we obtain
\begin{equation}\label{product 1- lambda case 1}
\prod_{n=1}^{k-1}|1-\lambda ^n|=\left |N !\right ||1-\lambda
^m|^{N }.
\end{equation}
By Lemma \ref{lemma order of p in factorial} we have
\[
\nu(N!)\leq \frac{N-1}{p-1}\leq \frac{k-1}{m(p-1)}
-\frac{1}{p-1},
\]
where each inequality become an equality if $(k-1)/m$ is a
non-negative
integer power of $p$. Together with (\ref{product 1- lambda case
1}) and the fact that $N=(k-1)/m-\delta(m)$ it follows that
\begin{equation}\label{prod s=0}
\prod_{n=1}^{k-1}|1-\lambda ^n| 
\geq
p^{\frac{1}{p-1}}|1-\lambda^m|^{-\delta(m)}\tilde{r}_1(\lambda)^{k-1},
\end{equation}
with equality if $(k-1)/m$ is a non-negative integer power of $p$.
Identifying $\tilde{r}(\lambda)$ with $\tilde{r}_1(\lambda)$
and using that $\delta(m)=0$ if $m$ divides $k-1$,
this completes the case $s=0$.

Now, suppose that $s\geq 1$.
By Lemma \ref{lemma distance 1-lambda m char 0,p} we have
\begin{equation}\label{1-lambda}
\left |1 -\lambda ^n \right |= \left\{\begin{array}{ll}
1, & \textrm{if \quad $m\nmid n$,}\\
|1-\lambda ^m|^{p^{\nu(n)}}, & \textrm{if \quad $m\mid n$ but
$mp^{s}\nmid n$,}\\
|n|p^t|\gamma_0-\lambda ^m|^{p^t}, &
\textrm{if \quad $m p^{s} \mid n$.}
\end{array}\right.
\end{equation}
Throughout the rest of this proof we put
\[
M:=\left\lfloor\frac{k-1}{mp^s}\right\rfloor, 
\quad \textrm{and } S_{k-1}:=\{ i\in[1,k-1]: mp^s\mid i\}.
\]
Note that so defined the cardinality $\#S_{k-1}=M$ and
$S_{k-1}=\{jmp^s:j\in[1,M]\}$. 
It follows from these observations that we have the following two identities
\begin{equation}\label{product |n| mp^s|n}
\prod_{n\in S_{k-1}}|n|= |M!|p^{-sM}, \quad\textrm{and}
\end{equation}
\begin{equation}\label{product p^s mp^s|n}
\prod_{n\in S_{k-1}}p^t|1-\lambda ^m|^{p^t}=
p^{tM}|1-\lambda ^m|^{p^tM}.
\end{equation}
Moreover, put
\[
\quad T_{k-1}:=\{i\in[1,k-1]: m\mid i \textrm{ but }
mp^s\nmid i \},
\]
and let $\Sigma$ be defined as in Lemma \ref{lemma Sigma}. 
Then, 
\begin{equation}\label{product p^nu p^s}
\prod_{
n\in T_{k-1} }|1-\lambda
^m|^{p^{\nu(n)}}= |1-\lambda ^m|^{\Sigma}.
\end{equation}
Using (\ref{1-lambda}) and multiplying the products
(\ref{product |n| mp^s|n}),
(\ref{product p^s mp^s|n}), 
and (\ref{product p^nu p^s})
we obtain
\begin{equation}\label{product 1 -lambda case 3}
\prod_{n=1}^{k-1}|1-\lambda ^n|=|M!|p^{-(s-t)M}|1-\lambda ^m|^{\Sigma}
|\gamma_0 -\lambda ^m|^{p^tM}.
\end{equation}
This equation is the starting point of the estimates given below.
By Lemma \ref{lemma order of p in factorial} we have
\begin{equation}\label{estimate M}
\nu (M!)\leq \frac{M-1}{p-1}\leq
\frac{k-1}{mp^s(p-1)} -\frac{1}{p-1},
\end{equation}
where each inequality become an equality if $(k-1)/mp^{s}$ is a
non-negative
integer power of $p$.
In view of Lemma \ref{lemma Sigma} we have
\begin{equation}\label{estimate Sigma_1}
\Sigma\leq
s\frac{k-1}{m}\frac{p-1}{p} + \delta(mp^s)p^{s-1},
\end{equation}
with equality if $mp^{s-1}$ is a divisor of $k-1$. 
Also recall that by definition
\[
\tilde{r}(\lambda)=R(s+1)^{\frac{1}{m}}p^{-\frac{s-t}{mp^s}}
|1-\lambda^m|^{s\frac{p-1}{mp} }
|\gamma_0-\lambda ^m|^{1/mp^{s-t}}.
\]
Applying
(\ref{estimate Sigma_1}) and (\ref{estimate M})
to the identity (\ref{product 1 -lambda case 3}),
and using $M=(k-1)/mp^s-\delta(mp^s)$, we obtain 
the following estimate
\begin{equation}\label{product inequality general}
\prod_{n=1}^{k-1}|1-\lambda^n|\geq
p^{\frac{1}{p-1}}
p^{(s-t)\delta(mp^s)}|1-\lambda^m|^{
\delta(mp^s) p^{s-1}
}
|\gamma_0-\lambda^m|^{-p^t\delta(mp^s)}
\tilde{r}(\lambda)^{(k-1)},
\end{equation}
with equality if $(k-1)/mp^{s}$ is a
non-negative
integer power of $p$. 
As $\delta(mp^s)=0$  if $mp^{s}$ is a divisor of $k-1$, 
we obtain the equality 
(\ref{prod eq p power case}) as required.

As to the inequality (\ref{product ineq general case}), 
first note that $p^{(s-t)\delta(mp^s)}\geq 1$.
Second, note that for $s\geq 1$ we have 
\[
|1-\lambda^m|^{p^{s-1}}\geq p^{-1/(p-1)},
\]
and $\delta(mp^s)\leq 1-1/mp^s$.
Together with (\ref{product inequality general}) 
this implies
\begin{equation}\label{product inequality general R(1)}
\prod_{n=1}^{k-1}|1-\lambda^n|\geq
p^{\frac{1}{mp^s(p-1)}}
|\gamma_0-\lambda^m|^{-p^t\delta(mp^s)
}
\tilde{r}(\lambda)^{k-1}.
\end{equation}
This completes the proof of the inequality
(\ref{product ineq general case}).

Concerning the inequality (\ref{prod ineq general case m=1 pnmid k-1}),
recall that $\gamma_0$ is minimizing $|\gamma_0-\lambda^m|$.
Hence, the case $s=1$ is a direct consequence of 
(\ref{product inequality general R(1)})
putting $m=1$ and using 
$|\gamma_0-\lambda|^{p^t}\leq |1-\lambda |$. 
As to the case $s\geq 2$,
suppose that 
$p^s\nmid k-1$ and put $\beta=\nu(k-1)$.
Then for all integers $j\in[\beta+1,s]$, we have $\delta(p^j)\geq 1/p^{j}$.
Hence, together with (\ref{Sigma_1 total}),  
for $m=1$, $s\geq 2$ and $p^s\nmid k-1$,
we obtain
the improved estimate
\begin{equation}\label{inequality sigma_1 s>1}
\Sigma \leq
s(k-1)\frac{p-1}{p} +\delta(p^s) p^{s-1} - (\beta-s)\frac{p-1}{p}.
\end{equation}
Also note that $\delta(p)\leq (p-1)/p$.
Accordingly, 
using the improved estimate (\ref{inequality sigma_1 s>1})
of $\Sigma$, 
we obtain 
the following improvement of (\ref{product inequality general R(1)}),
\[
\prod_{n=1}^{k-1}|1-\lambda^n|\geq
|1-\lambda |^{ 
-\delta(p)}
R(s+1)
|\gamma_0-\lambda|^{-p^t
\delta(p^s)
}
\tilde{r}(\lambda)^{k-1}.
\]
This completes the
proof of (\ref{prod ineq general case m=1 pnmid k-1}) and hence of the lemma.

\end{proof}

\begin{remark}
If $\gamma_0=1$, then by (\ref{prod s=0}) and
(\ref{product inequality general}),  
$\prod_{n=1}^{k-1}|1-\lambda^n|\geq R(1) \tilde{r}(\lambda)^{-(k-1)}$.
\end{remark}

\begin{proof}[Proof of Theorem \ref{theorem general estimate}] 
In view of Lemma \ref{lemma bk estimate indiff}
and Lemma \ref{lemma theorem general estimate} we have
\[
 ( \limsup |b_j|^{1/j}  )^{-1}\geq \tilde{r}(\lambda).
\]
This implies that $H_f$ converges on the open disk of radius
$\tilde{r}(\lambda)$.
Let $j\geq 2$ be an integer. 
By Lemma \ref{lemma bk estimate indiff} and
Remark \ref{remark strict} 
we have
\[
|b_j|\tilde{r}(\lambda)^j
<\tilde{r}(\lambda).
\]
It follows by Proposition \ref{proposition one-to-one} that
$H_f:D_{\tilde{r}(\lambda)}(0)\to D_{\tilde{r}(\lambda)}(0)$ is bijective.
Moreover,
according to the first statement of Proposition
\ref{proposition one-to-one}, the strict
inequality
\[
|b_j|\tilde{r}(\lambda)^j<
\tilde{r}(\lambda),
\]
implies that, if $H_f$ converges on the closed disk
$\overline{D}_{\tilde{r}(\lambda)}(0)$, then
$H_f:\overline{D}_{\tilde{r}(\lambda)}(0)\to
\overline{D}_{\tilde{r}(\lambda)}(0)$
is
bijective. Recall that as stated in Proposition \ref{proposition isometry},
$f:D_{1}(0)\to D_{1}(0)$ is a bijective isometry. Moreover,
$\tilde{r}(\lambda)<1$. Consequently, the linearization disk $\Delta_f$ includes
the disk $D_{\tilde{r}(\lambda)}(0)$ or $\overline{D}_{\tilde{r}(\lambda)}(0)$,
depending on whether or not, the conjugacy function $H_f$ converges on the
closed disk $\overline{D}_{\tilde{r}(\lambda)}(0)$.
This completes the proof of the theorem.

\end{proof}


\newpage

\section{The size of quadratic linearization disks}\label{section quadratic case}

The purpose of this section is to determine the exact radius
of linearization disks for quadratic polynomials and certain
power series perturbations of such maps.
Let $p$ be an odd prime and throughout this section let
\begin{equation}\label{definition of P II}
P_{\lambda}(x):=\lambda x +x^2\in\mathbb{C}_p[x], \quad \textrm{where $\lambda$
is not a root of unity and }  1/p<|1-\lambda|<1.
\end{equation}
We will prove Theorem \ref{thmA quadratic maps}
that states that the linearization disk $\Delta_{P_{\lambda}}$, is the open
disk $D_{r(P_{\lambda})}(0)$ of radius 
$r(P_{\lambda})=|1-\lambda|^{-1/p}\tilde{r}(\lambda)$.
Apart from Lemma \ref{lemma theorem general estimate},
the main ingredient in the proof of 
Theorem \ref{thmA quadratic maps} is the following lemma.

\begin{lemma}\label{lemma quadratic polynomials}
The coefficients of the
conjugacy $H_{P_{\lambda}}$ satisfy
\begin{equation}\label{bk quadratic simplified}
|b_{k}|=\frac{ 
|1-\lambda|^{ \left\lfloor \frac{k-1}{p}\right\rfloor }
} 
{\prod_{n=1}^{k-1}|1-\lambda^n|}, 
\quad \textrm{for } p\geq3, \textrm{and  } k\geq2.
\end{equation}
\end{lemma}

A key point 
in obtaining the 
exact absolute value of the coefficients of the corresponding 
conjugacy function $H_{P_{\lambda}}$
is that, using the assumptions
 $1/p<|1-\lambda|<1$ and $p\geq 3$, we prove that the coefficients
 of the conjugacy $\{b_k\}_{k\geq 1}$
 form a strictly increasing sequence in $k$.
 This may not be the case if we drop  
 the condition on $\lambda$ or $p$ respectively,
 since this  may yield cancellation of large terms
 (this will be explained later in 
 Remark  and \ref{remark p=2} and \ref{remark condition on lambda}).

Let $\mathbb{N}$ be the set of non-negative integers.
For $k\geq 2$ and $l\in[1,k-1]$, put
\begin{equation*}
A'_l(k):=\{(\alpha_1,\alpha_2)\in\mathbb{N}^2: \alpha_1+\alpha_2=l,
\textrm{ and } \alpha_1+2\alpha_2=k\}.
\end{equation*}
Given a rational number $x$, we denote by $\lceil x\rceil$ the smallest
integer greater than or equal to $x$.  Note that $A'_l(k)$ is non-empty 
if and only if $l\in[\lceil k/2\rceil,k]$. Given $k\geq 2$
and $l\in[\lceil k/2\rceil,k-1]$, the set $A'_l(k)$ of ordered pairs $(\alpha_1,\alpha_2)$
contains precisely one element,
\[
A'_l(k)=\{(2l-k,k-l)\}.
\]
For $k\geq 2$ and $l\in[\lceil k/2\rceil,k-1]$ we put
\begin{equation*}
C_l(k):=\frac{l!}{(2l-k)!(k-l)!}\lambda^{2l-k}.
\end{equation*}
It follows that for $k\geq2$, the coefficients of the conjugacy function
$H_{P_{\lambda}}(x)=\sum_{k\geq 1}b_kx^k$ satisfy the recurrence relation
\begin{equation}\label{bk_inC_l(k)}
b_k=\frac{1}{\lambda(1-\lambda^{k-1})}
\sum_{l=\lceil k/2\rceil}^{k-1}b_lC_l(k),
\end{equation}
starting with $b_1:=1$.

\begin{proof}[Proof of Lemma \ref{lemma quadratic polynomials}]
First note that $b_1=1$ and consequently
\[
b_{2}=\frac{1}{\lambda(1-\lambda)}.
\]
Hence, the lemma is clearly true for $k=2$ and $p\geq 3$.
We will proceed by induction in $k$.
Without loss of generality we may thus assume that the lemma
holds for $p\geq 3$ and $k-1\geq2$.
By the induction hypothesis
\begin{equation}\label{bk quadratic alternative}
|b_{k-1}|=\frac{|1-\lambda|^{\left\lfloor\frac{k-2}{p}\right\rfloor}}
{\prod_{n=1}^{k-2}|1-\lambda^n|}, \quad p\geq 3,
\textrm{ } k\geq3.
\end{equation}
Note that since $m=1$, by Lemma
\ref{lemma distance 1-lambda m char 0,p}, we have
\[
|1-\lambda^n|=|1-\lambda| \quad \textrm{if } p\nmid n.
\]
For future reference, also note that by the induction hypothesis
and the fact that for all integers $i \geq 1$,
\[
|1-\lambda |/|1-\lambda^{ip}|>1,
\]
we then have
\begin{equation}\label{b_k increasing}
|b_{k-1}|>|b_j|, \quad \textrm{for } j\in[1,k-2].
\end{equation}
Recall that since $C_l(k)$ is always an integer 
times a power of $\lambda$ we have
\begin{equation}\label{trivial bound C_l(k)}
|C_l(k)|\leq 1.
\end{equation}
Moreover, $A'_{k-1}(k)=\{(k-2,1)\}$ and hence
\[
C_{k-1}(k)=(k-1)\lambda^{k-2}.
\]
We identify two cases.

\textbf{Case I:} $p\nmid k-1$.
In this case
\[
|C_{k-1}(k)|=1.
\]
Hence, by the recurrence relation (\ref{bk_inC_l(k)}) and the facts
(\ref{b_k increasing}) and (\ref{trivial bound C_l(k)}),
using ultrametricity we obtain
\begin{equation}\label{b_k simple step}
|b_{k}|=\frac{|b_{k-1}|}{|1-\lambda^{k-1}|}.
\end{equation}
Moreover, as $p\nmid k-1$ we have $\lfloor (k-1)/p\rfloor=\lfloor (k-2)/p\rfloor$.
By (\ref{b_k simple step}) and the induction hypothesis (\ref{bk quadratic alternative})
we obtain (\ref{bk quadratic simplified}). This completes the induction step in this case.

\textbf{Case II:} $p\mid k-1$. First note that since $p\geq 3$, we may assume that
$k\geq 4$. Moreover, with the induction step for case I completed,
using (\ref{b_k increasing}) and (\ref{b_k simple step}),
replacing $k$ by $k-1$, we have
\begin{equation}\label{b_k increasing extended}
|b_{k-1}|=\frac{|b_{k-2}|}{|1-\lambda^{k-2}|}>|b_{k-2}|>|b_j|, \quad \textrm{for } j\in[1,k-3].
\end{equation}
Also note that by the assumption $p\mid k-1$ we have
\[
|C_{k-1}(k)|\leq |p|.
\]
Recall that by the condition of the lemma $|p|<|1-\lambda |$,
and hence
\[
|C_{k-1}(k)|<|1-\lambda|.
\]
By the assumption $p\mid k-1$
we have $|1-\lambda^{k-2}|=|1-\lambda|$,
and hence in view of (\ref{b_k increasing extended}) we obtain
\[
|C_{k-1}(k)b_{k-1}|<|b_{k-2}|.
\]
On the other hand, for $k\geq 4$, $A'_{k-2}(k)=\{(k-4,2)\}$ and
\begin{equation}\label{C_k-2(k)}
C_{k-2}(k)=\frac{(k-2)(k-3)}{2}\lambda^{k-4},
\end{equation}
and hence
\[
|C_{k-2}(k)|=1, \quad \textrm{since  } p\mid k-1, \textrm{ $p\geq 3$ and } k\geq 4.
\]
It follows that
\[
|C_{k-1}(k)b_{k-1}|<|b_{k-2}|=|C_{k-2}(k)b_{k-2}|.
\]
Then, by the recurrence relation (\ref{bk_inC_l(k)}) and the facts
 (\ref{trivial bound C_l(k)}) and (\ref{b_k increasing extended}),
using ultrametricity we have
\[
|b_{k}|=\frac{|b_{k-2}|}{|1-\lambda^{k-1}|}.
\]
Hence, using the induction hypothesis for $b_{k-2}$, we obtain
\[
|b_{k}|=\frac{|1-\lambda^{k-2}||1-\lambda|^{\left\lfloor\frac{k-2}{p}\right\rfloor}}
{\prod_{n=1}^{k-1}|1-\lambda^n|}.
\]
Again, since $p\mid k-1$ we have $|1-\lambda^{k-2}|=|1-\lambda|$ and
$\lfloor (k-1)/p\rfloor=\lfloor (k-2)/p\rfloor+1$. Accordingly,
\[
|1-\lambda^{k-2}||1-\lambda|^{\left\lfloor\frac{k-2}{p}\right\rfloor}=
|1-\lambda|^{\left\lfloor\frac{k-1}{p}\right\rfloor}.
\]
This completes the induction step for case II and together with
case I, it completes the proof of the lemma.

\end{proof}
\begin{remark}\label{remark p=2}
That the proof does not work for $p=2$ stems from the
fact that, in the case $k=3$, we then have $C_{k-2}(k)=0$ in (\ref{C_k-2(k)}). 
As a consequence, we obtain
\[
b_3=\frac{2}{1-\lambda}\cdot\frac{1}{\lambda(1-\lambda^2)}.
\]
Moreover, for larger $k$, we encounter a similar problem;
if $p=2$ and $2\mid k-1$, we may have that
$|C_{k-2}(k)|=|(k-2)(k-3)/2|$ is very small for large $k$.
\end{remark}
\begin{remark}\label{remark condition on lambda}
If the multiplier $\lambda $ satisfies $0<|1-\lambda|\leq 1/p$, 
then it is more complicated to calculate the absolute value 
of the coefficients of the conjugacy since in this case we may 
have cancellation of large terms. For example, consider the case $p=3$ and
\[
\lambda=1+3, \quad f(x)=\lambda x +x^2.
\]
Straightforward calculations of the coefficients of the conjugacy, 
starting with $b_1=1$, give
\[
b_2=1/[\lambda(1-\lambda)], \quad b_3=2/[\lambda(1-\lambda^2)(1-\lambda)],
\]
and
\[
b_4=\frac{1}{\lambda(1-\lambda ^3)}\left ( b_2+3\lambda^2 b_3\right).
\]
Note that
\[
|3|=|1-\lambda|=|1-\lambda^2|=\frac{1}{3},
\]
and hence
\[
|b_2|=|3\lambda^2 b_3|.
\]
Thus, by ultrametricity, the sum of these two terms in $b_4$ 
may add up to something small. In fact, we have
\[
 b_2+3\lambda^2 b_3 =
 \frac{1+5\lambda^2}{\lambda(1-\lambda^2)(1-\lambda)}
 =\frac{3^4}{4\cdot (5\cdot 3)\cdot 3}=\frac{3^2}{20}.
\]
Consequently, since $|1-\lambda^3|=3^{-2}$ we obtain
\[
|b_4|=1<|b_2|<|b_3|.
\]
Hence, it seems the coefficients of the conjugacy may not grow as fast as in the case
$1/p<|1-\lambda |<1$.
This example also indicates that, for large $k$, 
in this case it is a delicate problem to find the absolute value
 of the coefficients $b_k$  for $k$ large.
\end{remark}
\begin{lemma}\label{lemma quadratic polynomials tau}
Let $p\geq 3$ and $\tau=r(P_{\lambda})=|1-\lambda |^{-1/p}\tilde{r}(\lambda)$.
Then, for $k\geq 2$
the coefficients of
the conjugacy function $H_{P_{\lambda}}$ satisfy 
$|b_k|\tau^k<
\tau$. If $s\geq 1$ we have
\begin{equation}\label{quadratic estimate b_k}
|b_k|\tau^k
\leq 
R(s+1)\tau, 
\quad \textrm{for } k\geq 2.
\end{equation}
In particular, if  $k-1=p^{s+\alpha}$
for some integer $\alpha\geq 1$,
then $|b_k|\tau^k=p^{-\frac{1}{p-1}}\tau$.
\end{lemma}
\begin{proof}
Let $k\geq 2$ be an integer. By definition
\begin{equation*}\label{tau^k}
\tau^{-(k-1)}=|1-\lambda|^{\frac{k-1}{p}}
\tilde{r}(\lambda)^{-(k-1)}.
\end{equation*}
Hence, to prove the inequality (\ref{quadratic estimate b_k}) we need to prove that,
for $s\geq 1$ and $k\geq 2$,

\begin{equation}\label{need to prove}
|b_k|
\leq 
R(s+1)|1-\lambda|^{\frac{k-1}{p}}
\tilde{r}(\lambda)^{-(k-1)}.
\end{equation}
Note that by  Lemma \ref{lemma quadratic polynomials} we have the identity
(\ref{bk quadratic simplified}).
We have two cases. First, suppose that $p^s\mid k-1$. Then, by
(\ref{product ineq general case})
and (\ref{bk quadratic simplified}), 
using the assumption that $s\geq 1$,
we obtain (\ref{need to prove}).
Second, suppose that $p^s\nmid k-1$. 
Then, by  (\ref{prod ineq general case m=1 pnmid k-1}) and
(\ref{bk quadratic simplified}),
we obtain (\ref{need to prove}).
This completes the proof of the inequality
(\ref{quadratic estimate b_k}).

For $s=0$ we need to prove that 
$|b_k|\tau^k<\tau$ or equivalently
\begin{equation}\label{need to prove s=0}
|b_k|
< 
|1-\lambda|^{\frac{k-1}{p}}
\tilde{r}(\lambda)^{-(k-1)}.
\end{equation}
In view of 
(\ref{product ineq general case})
and (\ref{bk quadratic simplified}), 
we have
\[
|b_k|\leq |1-\lambda|^{\lfloor \frac{k-1}{p}\rfloor}p^{-1/(p-1)}\tilde{r}(\lambda)^{-(k-1)}.
\]
The inequality (\ref{need to prove s=0}) follows immediately if $p$ divides $k-1$.
If $p$ is not a divisor of $k-1$ we have 
$\lfloor (k-1)/p\rfloor \geq (k-1)/p-(p-1)/p$.
Also note that since $s=0$ we have $|1-\lambda|<p^{-1/(p-1)}$.
It follows that
\[
|b_k|<|1-\lambda|^{\frac{k-1}{p}}p^{1/p-1/(p-1)}\tilde{r}(\lambda)^{-(k-1)}.
\]
This completes the proof of (\ref{need to prove s=0}).

As to the last statement of the lemma,
given $k$ such that $ k-1=p^{s+\alpha}$,
for some integer $\alpha\geq 1$, we have 
the equality (\ref{prod eq p power case}) and hence by
Lemma \ref{lemma quadratic polynomials} we have
\[
|b_k|= p^{-\frac{1}{p-1}}\tilde{r}(\lambda)^{-(k-1)}|1-\lambda|^{\frac{k-1}{p}},
\quad \textrm{for $\alpha \geq 1$.}
\]
This completes the proof of the lemma.

\end{proof}

\begin{theorem}\label{theorem quadratic maps}
Let $p\geq 3$ and let $P_{\lambda}$ be a quadratic polynomial of the form
(\ref{definition of P II}).
Then, the linearization disk $\Delta_{P_{\lambda}}=D_{r(P_{\lambda})}(0)$, where
$r(P_{\lambda})=|1-\lambda|^{-1/p}\tilde{r}(\lambda)$.
\end{theorem}

\begin{proof}
First note that by Lemma \ref{lemma quadratic polynomials tau}
$H_{P_{\lambda}}$ converges on the open disk $D_{\tau}(0)$,
because $|b_k|\tau^k$ is bounded, 
and hence $|b_k|r^k\to 0$ for all $r<\tau$.
In fact, $H_{P_{\lambda}}$ diverges on the sphere $S_{\tau}(0)$;
let $I\geq s+1$ be an integer, then, by Lemma \ref{lemma quadratic polynomials tau}
we have
\[
|b_{p^I+1}|\tau ^{p^I+1}=
p^{-\frac{1}{p-1}}\tau,
\]
which does not approach zero as $I$ goes to infinity.
This proves that $H_{P_{\lambda}}$ diverges on the sphere
$S_{\tau}(0)$.
Finally, by Lemma \ref{lemma quadratic polynomials tau} we have
\[
|b_{k}|\tau ^{k}<\tau, \quad k\geq 2.
\]
Consequently, by Proposition \ref{proposition one-to-one},
$H_{P_{\lambda}}:D_{\tau}(0)\to D_{\tau}(0)$ is a bijection.
Recall that, as stated in Proposition \ref{proposition isometry},
$P_{\lambda}: D_{1}(0)\to D_{1}(0)$ is a bijective isometry.
Since $\tau<1$, it then follows that
the linearization disk of the quadratic polynomial $P_{\lambda}$ is the disk
$\Delta_{P_{\lambda}}=D_{\tau}(0)$.
This completes the proof of the theorem.

\end{proof}

\begin{corollary}\label{corollary quadratic power series}
Let $p\geq 3$ and $f\in\mathcal{G}_{\lambda}$, with $1/p<|1-\lambda|<1$, be of the form
\[
f(x)=\lambda x +\sum_{i = 2}^{\infty}a_ix^i\in\mathcal{O}_p[[x]], \quad
\textrm{where $|a_2|=1$ and $|a_2^2-a_3|=1$.}
\]
Then, the linearization disk $\Delta_{f}=D_{r(f)}(0)$, where
$r(f)=|1-\lambda|^{-1/p}\tilde{r}(\lambda)$.
\end{corollary}

\begin{proof}
By the conditions imposed on $a_2$ and $a_3$, the
same terms as in the proof of Lemma
\ref{lemma quadratic polynomials} will be strictly larger than all the others in
(\ref{bk-equation}). This follows since, for $f$ we obtain

1) $A_{k-1}(k)=\{(k-2,1,0,\dots,0)\}$, and hence we have
$C_{k-1}(k)=(k-1)\lambda^{k-2}a_2$ (with $|a_2|=1$),

2) $A_{k-2}(k)=\{(k-4,2,0,\dots,0),(k-3,0,1,0,\dots,0)\}$,
but since $|a_2^2-a_3|=1$ we then have
\[
|C_{k-2}(k)|=|\lambda^{k-4}a_2^2(k-2)(k-3)/2+\lambda^{k-3}a_3(k-2)|
=|\lambda a_2^2[(k-1)-2]/2+a_3|=1,
\]
for $p\geq 3$ and $p\mid k-1$. 
In the last equality we also use that $0<|1-\lambda|<1$.
Hence, all steps in the proof of Lemma \ref{lemma quadratic polynomials}
hold true replacing $P_{\lambda}$ by $f$,
and by the same arguments as in the proof of
Theorem \ref{theorem quadratic maps},
we obtain the desired result concerning the size of the linearization disk.
\end{proof}


\newpage

\section{Periodic points of quadratic-like maps
}\label{section corollary A}

The purpose of this section is to prove
Corollary \ref{corC no periodic boundary} that states
that for the quadratic polynomials and their power series
perturbations studied in the previous section,
the boundary of the linearization disk is free from 
periodic points. Let $p\geq 3$, and let $f$ be of the form
\begin{equation}\label{minimally ramified maps}
f(x)=\lambda x+\sum_{i=2}^\infty a_ix^i\in\mathcal{G}_{\lambda},
\textrm{ with $|1-\lambda|<1$,
$|a_2|=1$, and $|a_2^2-a_3|=1$.}
\end{equation}
For convenience, put 
\begin{equation}\label{Psi}
\Psi(\lambda):=|1-\lambda|^{-\frac{1}{p}}p^{-\frac{s-t}{p^s}}
|\gamma_0-\lambda|^{\frac{1}{p^{s-t}}}.
\end{equation}
We first show that, apart from $x=0$, a map 
$f$ of the form (\ref{minimally ramified maps}), 
has no periodic point in the
open disk $D_{\rho}(0)$, where
the radius $\rho=\min\{|1-\lambda|,\Psi(\lambda)\}$.
Second, we prove that if in addition $|1-\lambda|>1/p$, then
the linearization radius $r(f)<\rho$.
Apart from Corollary  \ref{corB quadratic power series},
the main ingredient in the proof
is the following lemma 
in which we localize
all periodic orbits of $f$.

\begin{lemma}\label{lemma localization of periodic points}
Let $p\geq 3$ and let $f$ be of the 
form (\ref{minimally ramified maps}).
Then,
the periodic points of $f$ in the open unit disk, are all of minimal
period $p^n$ for some integer $n\geq 0$, and
\begin{enumerate}

\item for $n=0$, the periodic fixed points are $x=0$ and
$x_0\neq 0$ with $|x_0|=|1-\lambda|$,

\item for $n\in[1,s-1]$ where $s\geq 2$, 
the periodic points of minimal period $p^n$ 
are located on the sphere of radius
$|1-\lambda|^{\frac{p-1}{p}}$, about the origin,

\item for $n= s$ where $s\geq1$, the periodic points 
of minimal period $p^n$ are located on the sphere of 
radius $\Psi(\lambda)$, about the origin, and

\item for $n\geq s+1$, the periodic points of minimal period $p^n$ 
are located on the sphere of radius $p^{-1/p^{n}}$, about the origin.
\end{enumerate}
In particular, apart from zero, $f$ has no periodic point in the
open disk $D_{\rho}(0)$, where
the radius $\rho=\min\{|1-\lambda|,\Psi(\lambda)\}$. 
Furthermore, $\rho=\Psi(\lambda)$ if and only if
$s-t\geq 2$ (or $s-t=1$ and $|\gamma_0-\lambda|\leq |1-\lambda|^2$).

\end{lemma}

To localize the periodic orbits we analyze the Newton polygons of iterates of $f$, using
the fact that all the periods of periodic points of $f$ are powers of the prime $p$, 
and that $f$ is minimally ramified
\cite[\textit{Exemple} 3.19]{Rivera-Letelier:2003thesis} as well as isometric.
The following proposition is known, 
see \textit{e.g.} \cite[p. 333]{Lubin:1994} or \cite[p. 190]{Rivera-Letelier:2003thesis}.

\begin{proposition}\label{proposition minimal period}
Let $f\in\mathcal{O}_p[[x]]$ be such that $f(0)=0$ and $|f'(0)-1|<1$.
Then, the minimal period of each periodic point of $f$ in
the open unit disk is a power of $p$.
\end{proposition}

For each integer $n\geq 0$, we will associate with
$f\in\mathcal{O}_p[[x]]$ 
such that $f(0)=0$ and $|f'(0)-1|<1$, the ramification number
\[
i_n:=\textrm{wideg}(f^{p^n}-\textrm{id})-1.
\]
So defined, $i_n$ counts the number of fixed points of $f^{p^n}$ in the open unit disk
different from zero, counting multiplicity.
By Sen's Theorem \cite{Sen:1969}, if
$i_n<\infty$,
then $i_n\equiv i_{n-1}\mod p^n$. Note that 
by the assumptions on $f$, $i_0\geq 1$ and $i_{n+1}\geq i_n +1 $
and together with Sen's Theorem this implies that 
$i_n\geq 1+p+\dots+p^n$.
The power series $f$ is said to be \textit{minimally ramified} if
\[
i_n=1+p+\dots+p^n, \quad \textrm{for all integers $n\geq 1$.}
\]
By a result of Keating \cite{Keating:1992},
$f$ is minimally ramified  if and only if
$i_0=1$ and $i_1=1+p$.
In fact, Rivera-Letelier 
\cite[\textit{Exemple} 3.19]{Rivera-Letelier:2003thesis} 
classified minimally ramified power series.

\begin{proposition}[See \cite{Rivera-Letelier:2003thesis}]\label{proposition minimal}
A power series $f(x)=\sum_{i\geq 1}a_ix^i\in\mathcal{O}_p[[x]]$ with 
 $|f'(0)-1|<1$, 
is minimally ramified if and only if $p\geq 3$, $|a_2|=1$ and $|a_2^2-a_3|=1$.
\end{proposition}

Let $g\in\mathcal{O}_p[[x]]$ be a power series of finite Weierstrass degree of the form
$g(x)=\sum_{j\geq 1} a_jx^{j}$.
We denote by $\mathcal{N}(g)$ the principal part of the Newton polygon of $g$,
\textit{i.e.} the convex hull in $\mathbb{R}^2$ of
\[
\mathcal{D}(g)=\{(j,\nu(a_j)):j\in\{1,\dots,\textrm{wideg}(g)\} \}.
\]
Recall that if a line segment of the Newton polygon $\mathcal{N}(g)$ has slope
$\kappa$ and the projection of the
segment to the $j$-axis has length $l$, then $g$ has exactly $l$ roots
of absolute value $p^{\kappa}$, counting multiplicity.
For all integers $n\geq 1$, we associate with 
$\lambda $ the number
\[
\kappa_n:=-\frac{\nu(1-\lambda^{p^n})-\nu(1-\lambda^{p^{n-1}})}{p^n}.
\]
Note that $\kappa_n$ is well defined since by assumption, $\lambda$
is not a root of unity.  

\begin{lemma}\label{lemma_oneslope}
Let $p\geq 3$ and let $f$ be of the form (\ref{minimally ramified maps}).
Then, for every integer
$n\geq 1$, the  Newton polygon $\mathcal{N}(f^{p^{n}}-\textrm{id})$
is obtained from that of $\mathcal{N}(f^{p^{n-1}}-\textrm{id})$ by adding only
\emph{one} single line segment of length
$p^n$ and slope $\kappa_n$. In particular, all periodic points of $f$
of minimal period
$p^n$ are located on the sphere of radius $p^{\kappa_n}$ about the origin.
\end{lemma}
\begin{proof}
By Proposition \ref{proposition minimal}
$f$ is minimally ramified.
Hence, for all integers $n\geq 1$, the  Newton polygon
$\mathcal{N}(f^{p^{n}}-\textrm{id})$
is obtained from that of $\mathcal{N}(f^{p^{n-1}}-\textrm{id})$
by adding line segments of total length
$\textrm{wideg}(f^{p^n}-\textrm{id})-
\textrm{wideg}(f^{p^{n-1}}-\textrm{id})=p^n$.
That the average weighted slope of these additional segments is $\kappa_n$,
then follows from the fact that
for all integers $n\geq 0$,
$f^{p^n}(x)=\lambda^{p^n}x+\langle x^2 \rangle$.
Also recall that by Proposition \ref{proposition minimal period},
for all periodic points of $f\in\mathcal{G}_{\lambda}$, with $0<|1-\lambda|<1$, the minimal
period is a power of $p$.
Consequently, since $f$ is minimally ramified,
we add only one single cycle going from the
$p^{n-1}$th to the $p^{n}$th iterate of $f$. Moreover,
as stated in Proposition \ref{proposition isometry}, $f$
is isometric on the open unit disk.
Hence, all the periodic points of such a cycle,
of length $p^n$, must be
located on the same sphere about the origin. If follows that this
sphere has to be of radius $p^{\kappa_n}$. This also proves that
in fact we add only one single line segment.
\end{proof}

\begin{lemma}\label{lemma kappa}
Let $p\geq 3$ and let $f$ be of the form (\ref{minimally ramified maps}).
Then,
\begin{enumerate}
\item for $n\geq s+1$, we have \quad
$\kappa_n=-1/p^n$,
\item for $n=s$ where $s\geq 1$, \quad
$\kappa_n=
-(s-t)p^{-s}-p^{-(s-t)}\nu(\gamma_0-\lambda)
+\frac{1}{p}\nu(1-\lambda)$,
and
\item for $n\in[1,s-1]$ where $s\geq 2$, we have \quad
$\kappa_n=-\frac{p-1}{p}\nu(1-\lambda)$.
\end{enumerate}
\end{lemma}

\begin{proof}
By assumption, $0<|1-\lambda|<1$.
Hence, by Lemma \ref{lemma distance 1-lambda m char 0,p},
for all integers $k\geq s$,
\[
|1-\lambda^{p^k}|=p^t|p^k| |\gamma_0-\lambda|^{p^t}.
\]
Recall that by definition,
 for all non-zero elements $x,y\in\mathbb{C}_p$, we have $\nu(x)=-\log_p|x|$ and
$\nu(xy)=\nu(x)+\nu(y)$.
Consequently, for every integer $n\geq s+1$, we obtain
\[
\nu(1- \lambda^{p^n}) - \nu(1-\lambda^{p^{n-1}}) = \nu(p^n)-\nu(p^{n-1})=1.
\]
This proves the first statement of the lemma.
As to the second statement note that for $k=s$,
\[
|1-\lambda^{p^k}|=p^{-(s-t)}|\gamma_0-\lambda|^{p^t}.
\]
Also note that if $k\in[0,s-1]$,
then by
Lemma \ref{lemma distance 1-lambda m char 0,p} we have
\[
|1-\lambda^{p^k}|=|1-\lambda|^{p^{k}}.
\]
Hence, for $n=s$ we obtain
\[
\nu(1-\lambda^{p^n})-\nu(1-\lambda^{p^{n-1}})
=
(s-t)+p^t\nu(\gamma_0-\lambda)-p^{s-1}\nu(1-\lambda).
\]
Finally, concerning the third statement,
provided that $s\geq 2$, for $n\in[1,s-1]$ we obtain
\[
\nu(\lambda^{p^n}-1)-\nu(1-\lambda^{p^{n-1}})
=
(p^{n}-p^{n-1})\nu(1-\lambda).
\]
This completes the proof of the lemma.
\end{proof}

\begin{proof}[Proof of Lemma \ref{lemma localization of periodic points}]
The fact that the minimal periods are all prime powers is a consequence
of Proposition \ref{proposition minimal period}. The absolute value of the fixed point $x_0\neq 0$
is $|1-\lambda|$, since by assumption we have $i_0=1$.
By the assumptions on $f$,
Lemma \ref{lemma_oneslope} and Lemma \ref{lemma kappa} apply,
and we obtain the second, third and fourth statement of the lemma.
As to the statements concerning $\rho$, first note that by the definition of $s$ we have
\[
|1-\lambda|<R(s+1)=p^{-1/p^s(p-1)},
\]
and hence
\[
|1-\lambda|<|1-\lambda|^{\frac{p-1}{p}}<p^{-1/p^{s+1}}.
\]
In addition, by definition
\[
p^{-(s-t)/p^s}=R(s)^{(s-t)\frac{p-1}{p}}\leq |1-\lambda|^{(s-t)\frac{p-1}{p}},
\]
so that
\[
\Psi(\lambda)
\leq 
|1-\lambda|^{\frac{(s-t)(p-1)-1}{p}}
|1-\lambda|^{\frac{1}{p^{s-t}}}
\leq  |1-\lambda|^{\frac{p-1}{p}},
\]
where the inequalities become an equality if and only if
$\gamma_0=1$ (so that $t=s$).  
Accordingly, apart from zero, $f$
has no periodic point in $D_{\rho}(0)$ where
$\rho=\min\{|1-\lambda|,\Psi(\lambda)\}$.

Concerning the last statement of the lemma, 
for the equality $\rho=\Psi(\lambda)$ to hold it is
necessary that $\gamma_0\neq 1$. In other words
we must have $s-t>0$, so that $s\geq 1$
and $|1-\lambda|=R(s)$.
In this case we have $p^{-1/p^s}=|1-\lambda|^{(p-1)/p}$ 
which implies
\[
\Psi(\lambda)=
|1-\lambda|^{\frac{(s-t)(p-1)-1}{p}}
|\gamma_0-\lambda|^{\frac{1}{p^{s-t}}}.
\]
As $|\gamma_0-\lambda|<|1-\lambda|$ we certainly 
have that
$\Psi(\lambda)<|1-\lambda|$ for $s-t\geq 2$.
If $s-t=1$ we have $\Psi(\lambda)\leq|1-\lambda|$
if and only if $|\gamma_0-\lambda|\leq|1-\lambda |^2$.
The latter is a solution since $p\geq 3$ 
and for $s-t=1$
we have $R(s-1)\leq|\gamma_0-\lambda|<R(s)=|1-\lambda|$.
This completes the proof of the lemma. 
\end{proof}

\begin{proof}[Proof of Corollary \ref{corC no periodic boundary}]
Let $p\geq 3$, $f$ be of the form (\ref{minimally ramified maps}), and
suppose $1/p<|1-\lambda|<1$.
By Corollary  \ref{corB quadratic power series},
the radius of the  linearization disk $\Delta_f$ is given by
$r(f)=|1-\lambda|^{-1/p}\tilde{r}(\lambda)$. By the assumptions on $\lambda$,
$m=1$, and
\begin{equation*}
r(f)=|1-\lambda|^{-1/p}R(s+1)p^{-\frac{s-t}{p^s}}|1-\lambda|^{s\frac{p-1}{p}}
|\gamma_0-\lambda |^{1/p^{s-t}}.
\end{equation*}
Suppose that $s\geq 1$. Then, $R(s+1)=R(s)^{1/p}$ and 
\begin{equation}\label{bound quadratic linearization}
r(f) = R(s)^{\frac{1}{p}}|1-\lambda|^{\frac{p-1}{p}s}\Psi(\lambda). 
\end{equation}
By definition $R(s)\leq |1-\lambda|$.
Consequently, for $s\geq 1$ we obtain
\begin{equation}\label{bound r(P)}
r(f) \leq |1-\lambda|\Psi(\lambda), 
\end{equation}
with equality if and only if $s=1$ and $|1-\lambda|=R(1)$.
By Lemma \ref{lemma localization of periodic points}, apart from zero,
$f$ has no periodic point in the open disk $D_{\rho}(0)$,
where the radius
$\rho=\min\{|1-\lambda|,\Psi(\lambda)\}$.
Together with the observation (\ref{bound r(P)}),
this  completes the proof of the corollary for the case $s\geq 1$.

Finally, if $s=0$ we have $1/p<|1-\lambda|<R(1)$
and by Lemma \ref{lemma closest root of unity}
$\gamma_0=1$ and $s=t$. 
Accordingly,  $r(f)=|1-\lambda|^{-1/p} R(1)|1-\lambda|$.
Note that $p^{1/p}R(1)=R(2)$, and hence by the assumption
$|1-\lambda|>1/p$ we obtain $r(f)<R(2)|1-\lambda|$. As
$\rho=|1-\lambda|$ for $s=0$, this completes the proof
of the corollary.
\end{proof}


\subsection{Example of linearization disk and periodic points of a quadratic map}\label{section example quadratic}

Let $p=3$ and put
\[
\tilde{P}(x):=\lambda x +x^2\in\mathbb{C}_3, \quad\textrm{with }
\lambda:=1+3^{\frac{1}{4}}.
\]
Apart from zero, $\tilde{P}$ has a fixed point of absolute value
$|1-\lambda|=3^{-1/4}$.
Note that $R(1)=3^{-1/2}$, $R(2)=3^{-1/6}$, $R(1)<|1-\lambda|<R(2)$,
and hence $s=1$ and,  in view of
Lemma \ref{lemma distance 1-lambda m char 0,p},
$\gamma_0=1$.
By (\ref{Psi}) we then have $\Psi(\lambda)=3^{-\frac{1}{6}}$,
and hence by (\ref{bound quadratic linearization}),
the radius of the corresponding linearization disk
$r(\tilde{P})=3^{-1/2}$.
By the example of Keating \cite[p. 321]{Keating:1992}
(and more generally Rivera-Letelier \cite[p. 191]{Rivera-Letelier:2003thesis}), $\tilde{P}$
is minimally ramified and Lemma
\ref{lemma localization of periodic points} applies.
The principal part of the corresponding Newton
polygon for the 9th iterate,
$\mathcal{N}(\tilde{P}^9-\textrm{id})$,
has three segments shown in figure \ref{newton polygon of P^3(x)-x}.
The distribution of the corresponding periodic points,
\textit{i.e.} roots of $\mathcal{N}(\tilde{P}^9-\textrm{id})$,
outside the linearization disk $\Delta_{\tilde{P}}$ is illustrated in figure
\ref{figure periodic P}.

\begin{remark}\label{remark post-critical}
In the complex field case \cite{Carleson/Gamelin:1991,Milnor:2000}, the boundary of the linearization disk,
$\partial \Delta_f$,
is contained in the closure of the post-critical set, the union of all
forward images $f^k(c)$, where $k\geq 1$ is an integer
and where $c$ ranges over all critical points of $f$. With
$\tilde{P}$ as above,
$c=-\lambda/2\in S_1(0)$,
$\tilde{P}(c)=-\lambda^2/4\in S_1(0)$, and
$\tilde{P}^2(c)=\lambda^3(\lambda-4)/16\in S_{|\lambda -1|}(0)$. As $\tilde{P}$ is isometric
on $S_{|1-\lambda|}(0)$, the forward iterates will stay on this sphere for $k\geq 2$,
and hence,
the intersection between the post-critical set and the boundary of the linearization
disk is empty in this case.
\end{remark}

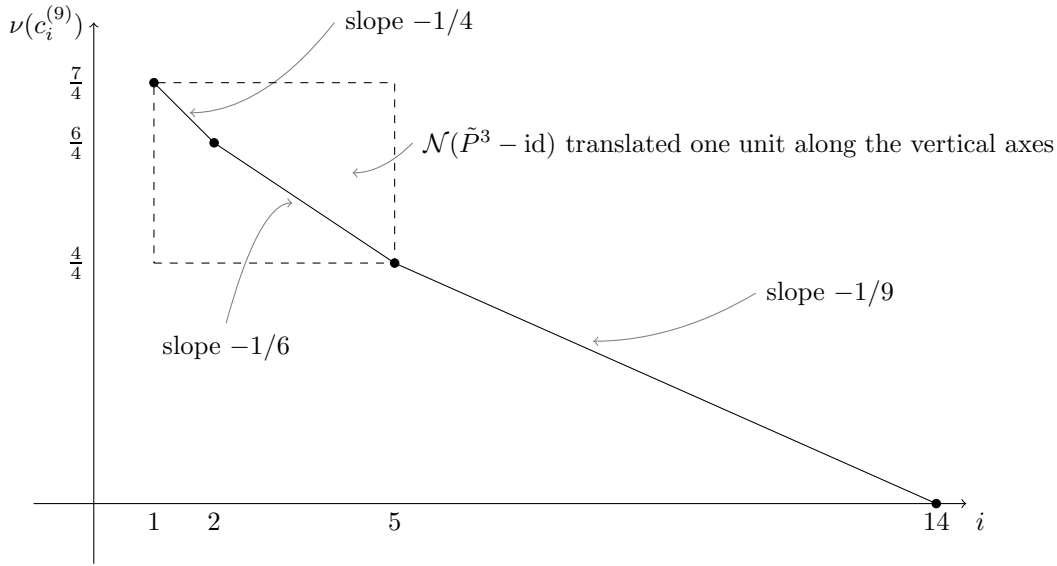
\begin{figure}
\begin{tikzpicture}[scale=0.8]

\draw[->] (-1,0) -- (14.5,0);
   \node [below right] at  (14.5,0) {$i$};
   \node [below] at (1,0) {1};
   \node [below] at (2,0) {2};
   \node [below] at (5,0) {5};
   \node [below] at (14,0) {14};

\draw[->] (0,-1) -- (0,8);
   \node[left] at (0,8) {$\nu(c_i^{(9)})$} ;
   \node[left] at (0,7) {$\frac{7}{4}$} ;
   \node[left] at (0,6) {$\frac{6}{4}$} ;
   \node[left] at (0,4) {$\frac{4}{4}$} ;

\draw (1,7) -- (2,6) -- (5,4) -- (14,0);


\draw [fill] (1,7) circle (0.07);
\draw [fill] (2,6) circle (0.07);
\draw [fill] (5,4) circle (0.07);
\draw [fill] (14,0) circle (0.07);

\draw[<-, gray] (1.6,6.5) parabola (4,8);
\node[right, ultra thick] at (4,8) {slope $-1/4$};

\draw[<-, gray] (3.3,5) parabola (2.2,3);
\node[below, ultra thick] at (2.2,3) {slope $-1/6$};

\draw[<-, gray] (8.3,2.7) parabola (11,3.5);
\node[right, ultra thick] at (11,3.5) {slope $-1/9$};

\draw[dashed] (1,7) rectangle (5,4);

\draw [<-, gray] (4.3,5.5) parabola (5.3,6);
\node[right, ultra thick] at (5.3,6)  {$\mathcal{N}(\tilde{P}^3-\textrm{id})$ translated one unit along the vertical axes} ;

\end{tikzpicture}
\caption{
The Newton polygon $\mathcal{N}(\tilde{P}^9-\textrm{id})$.
$\tilde{P}$ has a fixed point of absolute value
$|1-\lambda|=3^{-\frac{1}{4}}$, three periodic points of minimal period $3$ of absolute value
$\Psi(\lambda)=3^{-\frac{1}{6}}$, and nine periodic points of minimal period $9$ of absolute value
$3^{-\frac{1}{9}}$.
}\label{newton polygon of P^3(x)-x}
\end{figure}

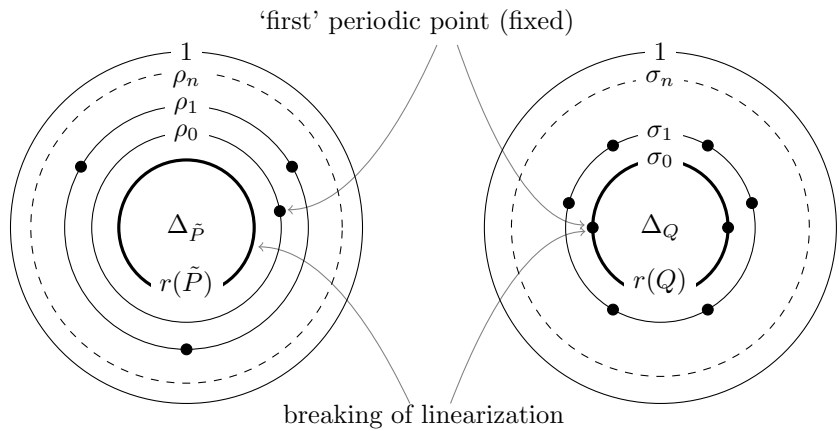
\begin{figure}
\begin{center}
\begin{tikzpicture}[scale=0.9]

\draw[very thick] (0,0) circle (1cm);
\node[fill=white] at (0,0) {$\Delta_{\tilde{P}}$};

\node[fill=white] at (0,-0.8) {$r(\tilde{P})$};

\draw[<-, gray] (-15:1.1)  parabola  (3.3,-2.6) ;
\node[above] at (3.5,-3.1) {breaking of linearization};
\draw[<-, gray] (7,0) +(183:1.11)  parabola  (3.9,-2.6) ;

\draw (0,0) circle (1.4cm);
\node[fill=white] at (0,1.4) {$\rho_0$};

\draw[fill] (10:1.4) circle (0.08);
\draw[<-,gray]  (10:1.52)  parabola (3.7,2.7);
\node[above] at (3.4,2.7) {`first' periodic point (fixed)};
\draw[<-,gray]  (7,0) +(178:1.12) parabola (4,2.7) ;

\draw (0,0) circle (1.8cm);
\node[fill=white] at (0,1.8) {$\rho_1$};

\draw[fill] (30:1.8) circle (0.08);
\draw[fill] (150:1.8) circle (0.08);
\draw[fill] (270:1.8) circle (0.08);

\draw[dashed] (0,0) circle (2.3cm);
\node[fill=white] at (0,2.2) {$\rho_n$};

\draw (0,0) circle (2.6cm);
\node[fill=white] at (0,2.6) {$1$};


\draw[very thick] (7,0) circle (1cm);
\node[fill=white] at (7,0) {$\Delta_{Q}$};

\node[fill=white] at (7,-0.8) {$r(Q)$};
\node[fill=white] at (7,1) {$\sigma_0$};

\draw[fill] (7,0)  +(180:1.0) circle (0.08);
\draw[fill] (7,0)  +(0:1.0) circle (0.08);

\draw (7,0) circle (1.4cm);
\node[fill=white] at (7,1.4) {$\sigma_1$};

\draw[fill] (7,0)  +(60:1.4) circle (0.08);
\draw[fill] (7,0)  +(120:1.4) circle (0.08);
\draw[fill] (7,0)  +(165:1.4) circle (0.08);
\draw[fill] (7,0)  +(240:1.4) circle (0.08);
\draw[fill] (7,0)  +(300:1.4) circle (0.08);
\draw[fill] (7,0)  +(15:1.4) circle (0.08);

\draw[dashed] (7,0)  circle (2.2cm);
\node[fill=white] at (7,2.2) {$\sigma_n$};

\draw (7,0) circle (2.6cm);
\node[fill=white] at (7,2.6) {$1$};

\end{tikzpicture}

\end{center}

\caption{
To the left, the linearization disk $\Delta_{\tilde{P}}(0)$
of radius $r(\tilde{P})=3^{-\frac{1}{2}}$
and periodic points;
$\rho_0=3^{-\frac{1}{4}}$ fixed;
$\rho_1=\Psi(\lambda)=3^{-\frac{1}{6}}$ minimal period $3$;
$\rho_n=3^{-\frac{1}{3^n}}$ ($n\geq2 $)  minimal period
$3^{n}$.
To the right,
 $Q(x)=(1+x)^{p+1}-1\in\mathbb{C}_p[x]$;
periodic points of minimal period $p^n$, $n\geq 0$, distributed on spheres of radius
$\sigma_n=p^{-1/(p^n(p-1))}$. The
linearization of $Q$ is broken by fixed points of absolute value
$r(Q)=\sigma_0$. In this case the conjugacy
$H_{Q}(x)=\log_p(1+x)$.
See \cite{Arrowsmith/Vivaldi:1994} for more details.
}\label{figure periodic P}

\end{figure}


\newpage

\section*{Acknowledgements}

I would like to thank Andrei Khrennikov and Juan Rivera-Letelier for fruitful
discussions and suggestions. I am also grateful for 
comments and corrections from the reviewers that certainly improved the
presentation.


\addcontentsline{toc}{section}{References}

\bibliographystyle{plain}


\end{document}